\documentclass[a4paper,10pt,twoside]{amsart}
\usepackage[latin1]{inputenc}
\usepackage[T1]{fontenc}
\usepackage{amsmath}
\usepackage{amsthm}
\usepackage{amssymb}
\usepackage[all]{xy}
\usepackage{epic}
\usepackage{stmaryrd}
\newtheorem{thm}{Theorem}[section]

\newtheorem{prop}[thm]{Proposition}

\newtheorem{lem}[thm]{Lemma}
\newtheorem{cor}[thm]{Corollary}
\newtheorem{conjecture}[thm]{Conjecture}

\numberwithin{equation}{section}

\theoremstyle{definition}
\newtheorem{definition}[thm]{Definition}
\newtheorem{remark}[thm]{Remark}

\newtheorem{ex}[thm]{Example}
\newtheorem{exer}[thm]{Exercise}

\newcommand{\qqed}{\hspace*{\fill}$\Box$}

\newcommand{\Db}{{\rm D}^{\rm b}}

\newcommand{\Aut}{{\rm Aut}}

\newcommand{\Pic}{{\rm Pic}}

\newcommand{\rk}{{\rm rk}}
\newcommand{\coh}{{\rm Coh}}

\newcommand{\End}{{\rm End}}
\newcommand{\Hom}{{\rm Hom}}

\newcommand{\ext}{{\rm ext}}

\newcommand{\Stab}{{\rm Stab}}

\renewcommand{\ker}{{\rm Ker}}

\newcommand{\Ext}{{\rm Ext}}

\newcommand{\cal}{\mathcal}
\newcommand{\ka}{{\cal A}}
\newcommand{\kb}{{\cal B}}
\newcommand{\kc}{{\cal C}}
\newcommand{\kd}{{\cal D}}

\newcommand{\kf}{{\cal F}}

\newcommand{\ko}{{\cal O}}
\newcommand{\kp}{{\cal P}}
\newcommand{\kr}{{\cal R}}
\newcommand{\kq}{{\cal Q}}
\newcommand{\kt}{{\cal T}}

\newcommand{\ZZ}{\mathbb{Z}}
\newcommand{\QQ}{\mathbb{Q}}
\newcommand{\RR}{\mathbb{R}}
\newcommand{\CC}{\mathbb{C}}

\newcommand{\HH}{\mathbb{H}}
\newcommand{\PP}{\mathbb{P}}

\renewcommand{\to}{\xymatrix@1@=15pt{\ar[r]&}}
\renewcommand{\rightarrow}{\xymatrix@1@=15pt{\ar[r]&}}
\renewcommand{\mapsto}{\xymatrix@1@=15pt{\ar@{|->}[r]&}}
\renewcommand{\twoheadrightarrow}{\xymatrix@1@=15pt{\ar@{->>}[r]&}}
\renewcommand{\hookrightarrow}{\xymatrix@1@=15pt{\ar@{^(->}[r]&}}
\newcommand{\congpf}{\xymatrix@1@=15pt{\ar[r]^-\sim&}}
\renewcommand{\cong}{\simeq}

\begin{document}
\title{Introduction to stability conditions}
\author{D. Huybrechts}
\address{Mathematisches Institut,
Universit{\"a}t Bonn, Endenicher Alle 60, 53115 Bonn, Germany}
\email{huybrech@math.uni-bonn.de}
\begin{abstract} These are notes of a course given at the `school on modu\-li spaces' at the Newton Institute in January 2011.
The abstract theory of stability conditions (due to Bridgeland and Douglas) on abelian and triangulated categories is developed via tilting and t-structures. Special emphasis is put on the bounded derived category of coherent sheaves
on smooth projective varieties (in particular for curves and K3 surfaces). The lectures were targeted at an audience with
little prior knowledge of triangulated categories and stability conditions but with a keen interest in vector bundles on curves.
\end{abstract} 
\maketitle
\setcounter{tocdepth}{1}
\tableofcontents

The title of the actual lecture course also mentioned derived categories prominently. And
indeed, the original idea was to give an introduction to the basic techniques used to study
$\Db(X)$, the bounded derived category of coherent sheaves on a (smooth projective)
variety $X$, and, at the same time, to acquaint the audience with the slightly technical notion
of stability conditions on $\Db(X)$ as invented by Bridgeland following work of Douglas.
The lectures were delivered in this spirit and this writeup tries to reflect the actual lectures, but
the emphasis has been shifted towards stability conditions considerably.
It seemed worthwhile to spend most of the lectures just on stability conditions
and to present some of the arguments used to study this new notion in detail.

These notes are meant to be a gentle introduction to stability conditions and
not as a survey of the area, although we collect a few pointers to the literature
in the last section. We start out by recalling stability for
vector bundles on curves and slowly move to the more abstract version provided
by stability conditions on abelian and triangulated categories. Roughly,
a stability condition on a triangulated category can be thought of as a refinement
of a bounded t-structure and we shall explain this relation carefully. Bridgeland endows
the space of all stability conditions with a natural topo\-logy. This gives rise to a completely
new kind of moduli space which has been much studied over the last years. 
The ultimate hope is that  a good understanding of the space of stability conditions
leads to a better grip on the category itself. The best example for this is
an intriguing conjecture of Bridgeland describing the group of autoequivalences of the
derived category of a K3 surface as a fundamental group of an explicit `period domain'
for the space of stability conditions.

Stability conditions on $\Db(X)$ will be studied for $X$  smooth and projective of dimension one
or two, but we will not touch upon the many results for $X$ only quasi-projective or for
more algebraic categories coming from quiver representations
and there are many more results and aspects that are not covered by these lectures, e.g.\ wall crossing
phenomena. Originally, stability conditions were invented in order to 
study $\Db(X)$ for projective Calabi--Yau threefolds, but up to this date
and inspite many attempts, not a single stability condition has been constructed in 
this situation. A glance at the discussion in the case of surfaces quickly shows why
this is so complicated, but see \cite{BMT,BBMT} for attempts in this direction.

The material covered in these lectures is based almost entirely on the two articles
\cite{BrAnn} and \cite{BrK3} by Bridgeland. Frequently, we add examples, hint 
at related results and highlight certain aspects, but we also take the liberty
to leave out unpleasant technical points of the discussion. Only  Section \ref{sec:K3}
contains material that is not completely covered by the existing literature. Here, the conjecture
of Bridgeland is rephrased in terms of classical moduli stacks and  their fundamental groups.

\smallskip

{\bf Acknowledgements:} I wish to thank the organizers of the `School on moduli spaces' at the Newton Institute for inviting me and the audience for a stimulating and demanding atmosphere. Parts of the material were also used for lectures on stability conditions at Peking University in  the fall of 2006 and at Ann Arbor in the spring of 2011. I would like to thank both institutions for their hospitality. I am grateful to Pawel Sosna and the referee for a careful reading of the first version and the many comments.

\section{Torsion theories and ${\rm t}$-structures}
This first lecture begins in Section \ref{sec:Recol} with a review of stability for vector bundles on algebraic curves which we will rephrase
in terms of phases in order to motivate the notion of a stability condition. Similarly, we first give examples for decomposing
the abelian category $\coh(C)$ of coherent sheaves on a curve $C$ and turn this later, in Section \ref{sect:tt}, into the abstract 
concept of a torsion theory for abelian categories. Its triangulated counterpart, t-structures, will be recalled as well, cf.\ Section
\ref{sect:tstr}. The final part of the first lecture
is devoted to the interplay between torsion theories and t-structures via tilting.

\subsection{$\mu$-stability on curves (and surfaces): Recollections}\label{sec:Recol}

Consider a smooth projective curve $C$ over an algebraically closed field $k=\bar k$. Let
$\coh(C)$ denote the abelian category of coherent sheaves on $C$, which we will consider with its
natural $k$-linear structure.

Recall that a coherent sheaf $E\in\coh(C)$ is called \emph{$\mu$-stable} (resp.\ \emph{$\mu$-semistable}) if
$E$ is torsion free (i.e.\ locally free) and for all proper subsheaves $0\ne F\subset E$ one
has $\mu(F)<\mu(E)$ (resp.\ $\mu(F)\leq\mu(E)$). Here, $\mu(~~)=\frac{\deg(~~)}{\rk(~~)}$
is the slope.
We shall rewrite this in terms of a stability function which is better suited for the more
general notion of stability conditions on abelian or even triangulated categories.

Define, $$Z(E):=-\deg(E)+i\cdot\rk(E).$$
Then, the \emph{phase} $\phi(E)\in(0,1]$ of a sheaf $0\ne E$  is defined uniquely by the condition
$$Z(E)\in\exp(i\pi\phi(E))\cdot\RR_{>0}.$$ 
The \emph{stability function} $Z$ defines a map
$$Z:\coh(C)\setminus\{0\}\to\overline\HH:=\HH\cup\RR_{<0}.$$
$$\hskip-3cm
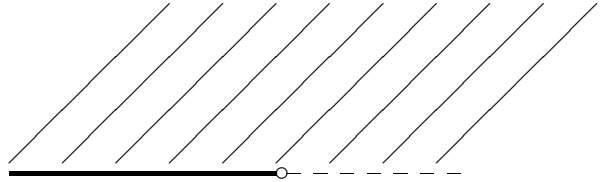
\begin{picture}(-100,100)
 {\linethickness{0.005mm}
 \put(-100,24){\line(1,1){60}}
 \put(-80,24){\line(1,1){60}}
\put(-60,24){\line(1,1){60}}
\put(-40,24){\line(1,1){60}}
\put(-20,24){\line(1,1){60}}
\put(0,24){\line(1,1){60}}
\put(20,24){\line(1,1){60}}
\put(40,24){\line(1,1){60}}
 \put(60,24){\line(1,1){60}}}
{  \linethickness{0.55mm}
\put(-100,20){\line(1,0){100}}
\put(2,20){\circle{4}}}
{  \linethickness{0.015mm}
\put(4,20){\line(1,0){5}}
\put(14,20){\line(1,0){5}}
\put(24,20){\line(1,0){5}}
\put(34,20){\line(1,0){5}}
\put(44,20){\line(1,0){5}}
\put(54,20){\line(1,0){5}}
\put(64,20){\line(1,0){5}}
}
\end{picture}
$$

\smallskip

\emph{Warning:} Later, the phase of an object in a triangulated category
is only well defined if the object is semistable or at least contained in the heart
of the associated t-structure.

\begin{remark}
i) If $\rk(E)>0$, e.g.\ when $E$ is locally free, then $Z(E)\in\HH$. More precisely, $Z(E)\in\RR_{<0}$
if and only if $E$ is torsion. In particular, $Z(k(x))=-1$.

ii) Note that $Z$ is additive, i.e.\ $Z(E_2)=Z(E_1)+Z(E_3)$ for any short exact sequence
$0\to E_1\to E_2\to E_3\to0$. Thus, it factorizes over the \emph{Grothendieck group}
$K(C)=K(\coh(C))$ of the abelian category $\coh(C)$, i.e.\
$Z:\coh(C)\to K(C)\to\CC$. The map $K(C)\to\CC$ is an additive group homomorphism.
\end{remark}

The following easy observation is important for motivating the notion of a stability condition later on.

\begin{lem}
Suppose $E\in\coh(C)$ is locally free. Then $E$ is $\mu$-stable if and only if for all proper subsheaves
$0\ne F\subset E$ the following inequality of phases holds true:
\begin{equation}\label{eqn:phase}
\phi(F)<\phi(E).
\end{equation}
\end{lem}
\begin{proof}
Since $E$ is locally free (and thus all non-trivial subsheaves $F\subset E$ are), we can divide by the rank.
Thus, $\frac{Z(E)}{\rk(E)}=-\mu(E)+i$ and then $\mu(F)<\mu(E)$ if and only if $-\mu(E)<-\mu(F)$ if and only
if $\phi(E)>\phi(F)$.

\vskip1cm

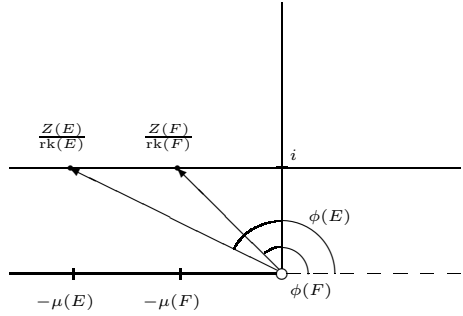
$$\hskip-3cm
\begin{picture}(-100,100)
{  \linethickness{0.055mm}
\put(-100,60){\line(1,0){169}}}
{  \linethickness{0.255mm}
\put(0,60){\line(1,0){4}}}
\put(5,63){\tiny\makebox{${i}$}}
{  \linethickness{0.0005mm}
\put(2,22){\line(0,1){100}}}
{  \linethickness{0.255mm}
\put(-100,20){\line(1,0){100}}
\put(2,20){\circle{4}}}
{ \linethickness{0.255mm}
\put(-76,18){\line(0,1){4}}}
\put(-90,8){\tiny\makebox{$-\mu(E)$}}
\put(-90,70){\tiny\makebox{$\frac{Z(E)}{{\rm rk}(E)}$}}
{ \linethickness{0.255mm}
\put(-36,18){\line(0,1){4}}}
\put(-50,8){\tiny\makebox{$-\mu(F)$}}
\put(-50,70){\tiny\makebox{$\frac{Z(F)}{{\rm rk}(F)}$}}
{  \linethickness{0.015mm}
\put(4,20){\line(1,0){5}}
\put(14,20){\line(1,0){5}}
\put(24,20){\line(1,0){5}}
\put(34,20){\line(1,0){5}}
\put(44,20){\line(1,0){5}}
\put(54,20){\line(1,0){5}}
\put(64,20){\line(1,0){5}}}
\put(2,30){
\qbezier(-6.3,-2.3)(-3.3,0.3)(0,0)}
\put(2,20){\oval(20,20)[rt]}
\put(5,12){\tiny\mbox{${\phi(F)}$}}
\put(2,30){
\qbezier(-17.8,-0.8)(-13.4,8.6)(0,10)}
\put(2,20){\oval(40,40)[rt]}
\put(12,40){\tiny\mbox{${\phi(E)}$}}
{ \linethickness{0.015mm}
\put(1,21.7){\vector(-1,1){38.3}}}
\put(-37,60){\circle*{2}}
{ \linethickness{0.025mm}
\put(0.2,21.2){\vector(-2,1){77.8}}}
\put(-77,60){\circle*{2}}
\end{picture}
$$
\vskip-0.5cm
\end{proof}

We leave it to the reader to rephrase $\mu$-semistability as a weak inequality for phases.

\begin{exer}
What happens if we forget about the assumption that $E$ is locally free?
Clearly, (\ref{eqn:phase}) for  arbitrary $E\in\coh(C)$ is equivalent to
$E$ being either a $\mu$-stable sheaf (and by definition in particular locally free) or
$E\cong k(x)$ for some closed point $x\in C$. The weak form of (\ref{eqn:phase})
is equivalent to $E$ being either a $\mu$-semistable sheaf (and in particular locally free)
or  a torsion sheaf.
\end{exer}

Thus, it seems natural to define (semi)stability (instead of $\mu$-(semi)sta\-bility) 
in the abelian category $\coh(C)$ in terms of the (weak) inequality (\ref{eqn:phase}).
It allows us to treat vector bundles and torsion sheaves on the same footing.
So from now on: {\bf \emph{Use  phases rather than slopes}}. 

\begin{exer}
Observe the useful formulae: $\mu(E)=-\cot(\pi\phi(E))$ and $\pi\phi(E)={\rm arcot}(-\mu(E))$.\
\end{exer}

Let us continue with the review of the classical theory of stable vector bundles on curves. The next step
consists of establishing the existence of \emph{Harder--Narasimhan} and \emph{Jordan--H\"older}
filtrations which filter any given sheaf such that the quotients are semistable resp.\ stable. 

Every $E\in\coh(C)$ admits a unique (Harder--Narasimhan) filtration
$0=E_0\subset E_1\subsetneq E_2\subsetneq\ldots\subsetneq E_n=E$ such that
the quotients $A_1:=E_1/E_{0},\ldots A_n=E_n/E_{n-1}$ are 
semistable sheaves of phase $\phi_1>\ldots >\phi_n$.

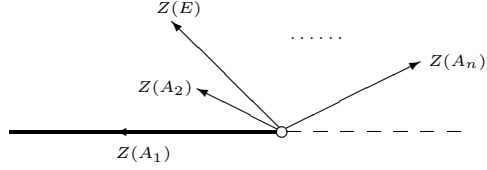
$$\hskip-3cm
\begin{picture}(-100,100)
 {\linethickness{0.005mm}
\put(1,21.6){\vector(-1,1){40}}
\put(-45,65){\tiny\mbox{$Z(E)$}}
\put(5,55){\tiny\mbox{$\ldots\ldots$}}
\put(0.4,21.2){\vector(-2,1){30}}
\put(-52,35){\tiny\mbox{$Z(A_2)$}}
\put(3.8,21.5){\vector(2,1){50}}}
\put(57,45){\tiny\makebox{$Z(A_n)$}}
\put(0,20){\vector(-1,0){60}}
\put(-60,10){\tiny\mbox{$Z(A_1)$}}

{  \linethickness{0.355mm}
\put(-100,20){\line(1,0){100}}
\put(2,20){\circle{4}}}
{  \linethickness{0.015mm}
\put(4,20){\line(1,0){5}}
\put(14,20){\line(1,0){5}}
\put(24,20){\line(1,0){5}}
\put(34,20){\line(1,0){5}}
\put(44,20){\line(1,0){5}}
\put(54,20){\line(1,0){5}}
\put(64,20){\line(1,0){5}}
}
\end{picture}
$$

Thus, $E_1$ is the torsion of $E$, which might  be trivial. To simplify notations,
we implicitly allow $E_1=0$, although $\phi_1$ would not be well defined in this case.
But the other inclusions are strict. The $A_i$ are called the \emph{semistable factors} of $E$. They are unique.

As a refinement of the Harder--Narasimhan filtration, one can also construct a (Jordan--H\"older) filtration
$0=E_0\subsetneq E_1\subsetneq E_2\subsetneq\ldots\subsetneq E_n=E$ of any $E\in\coh(C)$.
This time, the quotients $A_1:=E_1/E_{0},\ldots, A_n=E_n/E_{n-1}$ are 
stable sheaves of phase $\phi_1\geq\ldots \geq\phi_n$. 

Note that for $E$ not locally free, the first few $A_i$ are of the form $k(x_i)$, $x_i\in C$.
The $A_i$ are called the \emph{stable factors} of $E$.
The Jordan--H\"older filtration is in general not unique and the stable factors
are unique only up to permutation. 

\begin{exer}\label{rem:GP} Here are a few principles that hold true in the
general context of stability conditions on abelian or triangulated categories with identical proofs.

i) If $E,F\in\coh(C)$ are semistable (resp.\ stable) such that $\phi(E)>\phi(F)$ (resp.\ $\phi(E)\geq\phi(F)$), then $$\Hom(E,F)=0.$$

ii) If $E,F\in\coh(C)$ are stable such that $\phi(E)\geq\phi(F)$, then $$\text {either~}E\cong F \text{~or~}
\Hom(E,F)=0.$$

iii) If $E$ is stable, then $\End(E)\cong k$. (Recall $k=\bar k$.)
\end{exer}

\begin{exer}\label{exer:allsta}
If $E$ is semistable, then all stable factors $A_i$ of $E$ have the same phase $\phi(E)$.
Suppose $\Hom(A_{i_0},E)\ne0$ for some stable factor $A_{i_0}$ of $E$, then there exists a short exact sequence
$$0\to E'\to E\to E''\to 0$$
with $E',E''$ semistable of phase $\phi(E)$, such that all stable factors of $E'$ are isomorphic
to $A_{i_0}$ and $\Hom(E',E'')=0$ (or, equivalently, $\Hom(A_{i_0},E'')=0$).
Note that $E''=0$ if and only all $A_i$ are isomorphic to $A_{i_0}$.
\end{exer}

\begin{definition}\label{dfb:catsstPcurve}
For $\phi\in(0,1]$ one defines $$\kp(\phi):=\{E\in\coh(C)~|~\text{semistable of phase }\phi\},$$
which we will consider as a full linear subcategory of $\coh(C)$.
\end{definition}

Note that $\kp(1)\subset\coh(C)$ is the full subcategory of torsion sheaves. More generally, $\kp(\phi)\subset\coh(C)$
are full abelian subcategories of \emph{finite length}, i.e.\ ascending and descending chains of subobjects stabilize.
In fact, the \emph{minimal}\footnote{In representation theory, they would be called \emph{simple}, but this has another meaning
for sheaves. Also note that the minimal objects in the category $\coh(C)$ are the point sheaves $k(x)$, $x\in C$. In particular, $\coh(C)$
is not of finite length.}
objects in $\kp(\phi)$, i.e.\ those that do not contain any proper subobjects in $\kp(\phi)$,
are exactly the stable sheaves of phase $\phi$. The finite length of $\kp(\phi)$ is thus a consequence of the existence
of the finite(!) Jordan--H\"older filtration. In the general context, the finite length condition is tricky. Here,
it follows easily as the stability function $Z=-\deg+i\cdot\rk$ is rational (cf.\ Proposition \ref{prop:discZ}).

For any interval $I\subset (0,1]$ one defines $\kp(I)$ as the full subcategory of sheaves $E\in\coh(C)$ with (semi)stable
factors $A_i$ having phase $\phi(A_i)$ in $I$. They are only additive subcategories of $\coh(C)$ in general.

\vskip1cm
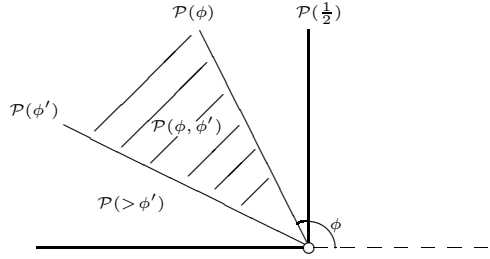
$$
\hskip-2cm
\begin{picture}(-100,100)
{  \linethickness{0.255mm}
\put(-100,20){\line(1,0){100}}
\put(1.3,22){\line(-1,2){40}}
\put(2,20){\circle{4}}
\put(0.4,21.3){\line(-2,1){90}}}
\put(-49,107){\tiny\makebox{$\kp(\phi)$}}
\put(-110,70){\tiny\makebox{$\kp(\phi')$}}
\put(-57,64){\tiny\makebox{$\kp(\phi,\phi')$}}
\put(-77,35){\tiny\makebox{$\kp(>\!\phi')$}}
{  \linethickness{0.21mm}
\put(1.8,22){\line(0,1){80}}}
\put(-2.8,107){\tiny\makebox{$\kp(\frac{1}{2})$}}
{  \linethickness{0.015mm}
\put(4,20){\line(1,0){5}}
\put(14,20){\line(1,0){5}}
\put(24,20){\line(1,0){5}}
\put(34,20){\line(1,0){5}}
\put(44,20){\line(1,0){5}}
\put(54,20){\line(1,0){5}}
\put(64,20){\line(1,0){5}}}
{\linethickness{0.15mm}
\put(2,30){\qbezier(-4.3,-1)(-3.3,0.1)(0,0)}
\put(2,20){\oval(20,20)[rt]}}
\put(10,27){\tiny\mbox{$\phi$}}
{  \thinlines
\put(-33,42){\line(1,1){15}}
\put(-68,58){\line(1,1){32}}
\put(-78,64){\line(1,1){36}}
\put(-23,36){\line(1,1){10}}
\put(-43,47){\line(1,1){19}}
\put(-40,68){\line(1,1){10}}
\put(-57,52){\line(1,1){10}}
}
\end{picture}
$$

\smallskip

The next remark will lead us naturally to the notion of a torsion theory.

\begin{remark}\label{rem:torsionCoh(C)}
i) Let us look at the following special case: For $\varphi\in(0,1]$ consider
 $\kt_\varphi:=\kp(\varphi,1]$ and $\kf_\varphi:=\kp(0,\varphi]$.
Then for all $E\in\coh(C)$ there exists a unique short exact sequence
$$0\to E'\to E\to E''\to0$$
with $E'\in\kt_\varphi$ and $E''\in\kf_\varphi$. The uniqueness follows from the observation
that $\Hom(\kt_\varphi,\kf_\varphi)=0$, i.e.\ there are no non-trivial homomorphisms from any object in $\kt_\varphi$ to
any object in $\kf_\varphi$. 
 
ii) Similarly, if we let $\kt:=\kp(1)$ and $\kf:=\kp(0,1)$, one obtains as above a short exact sequence
 with $E'$ being the torsion of $E$.
\end{remark}

\begin{remark}
 What happens if we pass from curves to surfaces? So let $S$ be a smooth projective surface with an ample 
(or big and nef) divisor $H$ viewed as an element of the real vector space
${\rm NS}(S)\otimes\RR$. Then ${\rm deg}_H(E)=({\rm c}_1(E).H)$ is well-defined and $\mu_H(E)={\rm deg}_H(E)/\rk(E)$
can be used to define $\mu$-stability for torsion free sheaves. A priori, one could now try to play the same
game with the function $Z_H=-{\rm \deg}_H+i\cdot\rk$, but there are immediate problems:
i) $Z_H(E)=0$ for sheaves supported in dimension zero, so $Z_H$ takes values in $\overline \HH\cup\{0\}$. ii) There is no Jordan--H\"older filtration for sheaves of dimension one, i.e.\ those supported on the curve. The rank (on $S$) of such a sheaf is
zero and its degree only reflects its rank as a sheaf on its support. Thus, for a curve $C\subset S$ all 
bundles on $C$ would be semistable of phase $1$ on $S$. (But see \cite{MP}, where it is shown
that the quotient of $\coh(S)$ by the subcategory of $0$-dimensional sheaves together with $Z_H$ behaves almost
like $\coh(C)$.)

Although the situation seems more complicated than for curves, it has one interesting
feature that is not present in the case of curves:
The  function $Z_H$ itself depends on a parameter, namely $H\in{\rm NS}(S)\otimes \RR$.
We will come back to this later.

Since it is known that on higher dimensional varieties one should rather work
with Gieseker stability than $\mu$-stability, i.e.\ taking the full Hilbert polynomial into account
for defining stability, one could try to adapt the approach here accordingly (see \cite{GR}), which taken
literally leads to a theory in which the stability function will take values in a higher dimensional
space and not simply  in $\CC=\RR^2$.
\end{remark}

\subsection{Torsion theories in abelian categories}\label{sect:tt}

The  abstract notion of a torsion theory has been first introduced by Dickson in \cite{Dick}, but was implicitly already present in earlier work of Gabriel.
The two standard references are \cite{HRS,BR}. In the following, we denote by $\ka$ an abelian category. Usually, $\ka$ is linear over some field $k$, but this will not be important for now.

The following notion is the abstract version  of the two examples in Remark
\ref{rem:torsionCoh(C)}. 

\begin{definition}
A \emph{torsion theory} (or \emph{torsion pair}) for $\ka$ consists of a pair of full subcategories $\kt,\kf\subset\ka$ such that:

i) $\Hom(\kt,\kf)=0$ and 

ii) For any $E\in\ka$ there exists a short exact sequence
\begin{equation}\label{eqn:sestorsion}
0\to E'\to E\to E''\to 0
\end{equation}
with $E'\in \kt$ and $E''\in\kf$.
\end{definition}

For obvious reasons, one calls the  (in general only additive) subcategories $\kt\subset\ka$ and $\kf\subset\ka$ the torsion part
resp.\ torsion free part of the torsion theory.

\begin{exer}
i) The short exact sequence (\ref{eqn:sestorsion}) is unique. For this, one 
uses $\Hom(\kt,\kf)=0$.

ii) The inclusion $\kt\subset\ka$ admits a right adjoint $T_\kt:\ka\to\kt$ given by mapping $E\in\ka$ to its torsion part
$T_\kt(E):=E'$ in (\ref{eqn:sestorsion}), i.e.\ for all $T\in\kt$ one has
$\Hom_\ka(T,E)=\Hom_\kt(T,T_{\kt}(E))$.

iii) Similarly, $\kf\subset\ka$ admits the left adjoint $\ka\to\kf$, $E\mapsto E/T_\kt(E)$, i.e.\
$\Hom_\ka(E,F)=\Hom_\kf(E/T_\kt(E),F)$ for all $F\in\kf$.

iv) The subcategories $\kt,\kf\subset\ka$ are closed under extensions. Moreover, $\kt$ is closed under quotients and $\kf$ is closed
under subobjects.

v) If $\kt^\perp$ is defined as the full subcategory of all objects $E\in\ka$ such that $\Hom(T,E)=0$ for all $T\in\kt$. Then
$\kf=\kt^\perp$. Similarly, $\kt=\!{}^\perp\kf$.
\end{exer}
\begin{remark}
In the same manner, one can prove the following useful fact (see e.g.\ \cite[Prop.\ 1.2]{BR}): For a full additive subcategory
$\kt\subset\ka$, which is closed under isomorphisms, $(\kt,\kt^\perp)$ defines a torsion pair if and only if $\kt\subset\ka$ admits a right
adjoint $T_\kt:\ka\to\kt$ and $\kt$ is closed under right exact sequences (i.e.\ if $T_1\to E\to T_2\to0$ is exact with $T_1,T_2\in\kt$, then also $E\in\kt$).
\end{remark}
In Section \ref{sec:tilts} we shall explain how to tilt an abelian category with respect to a given torsion theory. This is a construction that takes place naturally within the bounded derived category. For a more direct construction not using
derived categories, see Noohi's article \cite{No}. He associates to a torsion theory $(\kt,\kf)$ for an
abelian category directly a new abelian category $\kb$ the objects of which are of the form $[\varphi:E^{-1}\to E^0]$
with ${\rm ker}(\varphi)\in\kf$ and ${\rm coker}(\varphi)\in\kt$. The definition of morphisms in $\kb$ 
is a little more involved but still very explicit.  It is also shown in \cite{No} that $\kb$ describes a category that
is equivalent to the Happel--Reiten--Smal{\o} tilt to be discussed below.
\subsection{t-structures on triangulated categories}\label{sect:tstr}
In the following we shall denote by $\kd$ a triangulated category, e.g.\ $\kd$ could be the bounded derived category
$\Db(\ka)$
of an abelian category $\ka$. A triangulated category
$\kd$ is an additive (but not abelian!) category endowed with two additional structures: The \emph{shift functor}, an equivalence
$\kd\congpf\kd$, $E\mapsto E[1]$, and a collection of \emph{exact triangles} $E\to F\to G\to E[1]$ replacing short exact sequences
in abelian categories. They are subject to axioms TR 1-4, see e.g.\ \cite{GM,Hart,KS,Neem,Ver}.
E.g.\ one requires that with $E\to F\to G\to E[1]$ also 
$\xymatrix@C=1.5em{F\ar[r]& G\ar[r]& E[1]\ar[r]&F[1]}$ is an exact triangle. Another consequence of the axioms
is that for any such exact triangle
one obtains long exact sequences $$\Hom^{i-1}(A,G)\to\Hom^i(A,E)\to\Hom^i(A,F)$$ and
$$\Hom^i(F,A)\to\Hom^i(E,A)\to\Hom^{i+1}(G,A).$$ 

Here, $\Hom^i(A,B):=\Hom(A,B[i])=:\Ext^i (A,B)$.

In the following, we shall often simply write $E\to F\to G$ for an exact triangle $E\to F\to G\to E[1]$.
\begin{definition}\label{def:tstr}
A \emph{t-structure} on a triangulated category $\kd$ consists of a pair of full additive subcategories
$(\kt,\kf)$ such that:

i) $\kf=\kt^\perp$.

ii) For all $E\in\kd$ there exists an exact triangle
\begin{equation}\label{eqn:exact3}
E'\to E\to E''
\end{equation}
 with $E'\in\kt$ and $E''\in\kf$.

iii) $\kt[1]\subset\kt$.
\end{definition}

The first two conditions are reminiscent of the definition of  a torsion theory for abelian categories and we will comment
on this relation below. The last condition takes the triangulated structure into account. Note that we do not require
$\kt[1]=\kt$.\footnote{Note that $\kt=\kd$ and $\kf=0$ defines a t-structure, which of course satisfies
$\kt[1]=\kt$. But it is neither bounded nor very useful.}

To stress the analogy to torsion theories, we used the notation $\kt$ and $\kf$. More commonly however,
one writes $\kd^{\leq0}:=\kt$ and $\kd^{\geq1}=\kd^{>0}:=\kf$. With
$\kd^{\leq -i}:=\kd^{\leq0}[i]$ one gets $$\ldots\subset\kd^{\leq-2}\subset\kd^{\leq-1}\subset\kd^{\leq0}\subset\kd^{\leq1}\subset\ldots.$$

As for torsion theories, the inclusion $\kd^{\leq0}\subset\kd$ has a right adjoint $\tau^{\leq0}:\kd\to\kd^{\leq0}$,
$E\mapsto \tau^{\leq0}E:=E'$ (as in (\ref{eqn:exact3})). Similarly, $\tau^{\geq1}:\kd\to\kd^{\geq1}$ defines a left adjoint to the inclusion
$\kd^{\geq1}\subset\kd$. Thus, (\ref{eqn:exact3}) can be written as an exact triangle
$$\tau^{\leq0}E\to E\to\tau^{\geq1}E.$$

The \emph{heart} of a  t-structure is defined as $$\ka:=\kd^{\leq0}\cap\kd^{\geq0},$$
which in the notation of Definition \ref{def:tstr} is $\ka=\kt\cap\kf[1]$.
It is an abelian category and short exact sequences in $\ka$ are precisely the exact triangles
in $\kd$ with objects in $\ka$. Moreover, one has cohomology functors $H^i:\kd\to\ka$, $E\mapsto (\tau^{\geq i}\tau^{\leq i}E)[i]$.

A t-structure on a triangulated category $\kd$ is \emph{bounded} if  every $E\in\kd$ is contained in $\kd^{\leq n}\cap\kd^{\geq -n}$ for $n\gg0$. 

\begin{ex}
For $\kd=\Db(\ka)$ the standard bounded t-structure is given by $\kd^{\leq 0}:=\{ E~|~H^i(E)=0 ~i>0\}$. Then,
$\kd^{\geq 0}=\{ E~|~H^i(E)=0~i<0\}$
and its heart is the original abelian category $\ka\subset\kd$ in degree zero.
\end{ex}

\begin{remark}
i) Not every triangulated category $\kd$  admits a bounded t-structure. In fact, the existence of a bounded t-structure
on $\kd$  implies that $\kd$ is  \emph{idempotent closed} (or, \emph{Karoubian}). This is a folklore result (see \cite{Le}).
Recall that a morphism $e:E\to E$ in $\kd$ is idempotent if $e^2=e$. An idempotent morphism $e:E\to E$ is split if
$e=b\circ a:E\to F\to E$ with $a\circ b={\rm id}$ or, equivalently, if $E\cong F\oplus F'$ and $e$ is the composition of 
the natural projection and  inclusion  $E\twoheadrightarrow F\,\,\hookrightarrow E$. (For the equivalence of
the two descriptions the triangulated structure is needed.) The category $\kd$ is called idempotent
closed if every idempotent is split. For a given t-structure on $\kd$ any idempotent $e:E\to E$ yields
idempotent morphisms $\tau^{\leq 0}e:\tau^{\leq 0}E\to\tau^{\leq 0}E$ and $\tau^{> 0}e:\tau^{> 0}E\to\tau^{> 0}E$. 
One shows that $e$ is split if $\tau^{\leq0}e$ and $\tau^{>0}e$ are split. For a bounded t-structure this allows one to reduce
to objects in the (shifted) heart $\ka$. But clearly, due to the existence of kernels and cokernels in abelian categories,
any idempotent in $\ka$ splits.

ii) Suppose  $\kd'\subset\kd$ is a  triangulated subcategory. It is called \emph{dense} if every object $F\in\kd$ is a direct summand of an object in $\kd'$, i.e.\ there exists $F'\in\kd$ with $E\cong F\oplus F'\in\kd'$. 
If $F\not\in \kd'$, then the induced natural idempotent $E\to E$ splits only in $\kd$, but not in $\kd'$.  For a concrete
example, take the smallest full triangulated subcategory  $\kd'\subset \Db(\coh(C))$ that contains all line bundles
$\ko(np)$, where $C$ is a curve of genus $>0$ and $p\in C$ is a fixed closed point. 
In particular, $\kd'$ is dense in $\kd$ but itself not idempotent closed.
\end{remark}

\begin{remark}\label{rem:BrAnn}
For the following see \cite[Lemma 3.2]{BrAnn}. If $\kd$ is a triangulated category and $\ka\subset \kd$
is a full additive subcategory. Then $\ka$ is the heart of a bounded t-structure if and only if

i) $\Hom(\ka[k_1],\ka[k_2])=0$ for $k_1>k_2$ and

ii) For any $E\in \kd$ there exists a diagram
$$\xymatrix@C=1em{0=E_0~~~\ar[rr]&&E_1\ar[dl]\ar[rr]&&E_2\ar[dl]\ldots\ar[rr]&&E_{n-1}\ar[rr]&&E_n=E\ar[dl]\\
&A_1\ar[lu]^{[1]}&&A_2\ar[lu]^{[1]}&&&&A_{n}\ar[lu]^{[1]}&&}$$
where $E_{i-1}\to E_i\to A_i$ are exact triangles with $A_i\in\ka[k_i]$ and $k_1>\ldots>k_n$.

Note that due to i) the diagram is unique. If i) and ii) hold, then the corres\-ponding bounded t-structure is
defined by $\kd^{\leq0}:=\{E~|~A_i=0\text{ for }k_i>0\}$. Conversely, if the bounded t-structure is given, then
$E_1:=\tau^{\leq -k_1}(E)=A_1\in\ka[k_i]$, where $-k_1=\min\{ i~|~H^i(E)\ne0\}$. 
\end{remark}

In general, t-structures are not preserved by equivalences. Positively speaking, if $\Phi:\kd\congpf\kd$ is an exact autoequivalence 
of a triangulated cate\-gory and $\kd^{\leq0}\subset\kd$ defines a 
t-structure, then $\Phi(\kd^{\leq0})\subset\kd$ defines a new
 t-structure which is usually different from the original one.
\subsection{Torsion theories versus t-structures}\label{sec:tilts}

So far, we have seen t-struc\-tures as an analogue of torsion theories for abelian categories adapted to the triangulated setting. Another aspect of torsion theories for abelian categories is that they lead to new t-structures on triangulated categories.

Let us start with an abelian category $\ka$ and a torsion theory for it given by $\kt,\kf\subset\ka$.
Then consider the \emph{tilt} of $\ka$ with respect to $(\kt,\kf)$ defined
as the full additive subcategory $\ka^\sharp:=Tilt_{(\kt,\kf)}(\ka)\subset\Db(\ka)$ of all objects
$E\in\Db(\ka)$ with
$$H^0(E)\in\kt,~H^{-1}(E)\in\kf,\text{ and }H^i(E)=0\text{ for }i\ne0,-1.$$
Note that the definition makes sense in the more general context when $\ka$ is the heart
of a t-structure on a triangulated category $\kd$. In this case, the $H^i$ are the cohomology
functors that come with the t-structure.

For any $E\in\ka$ there exists a short exact sequence in $\ka$ or, equivalently, an exact triangle in $\kd$
of the form $T\to E\to F$ with $T\in\kt$ and $F\in\kf$. The latter corresponds to a class in $\Ext^1(F,T)$.
On the other hand, for $E\in\ka^\sharp$ there exists an exact triangle $F[1]\to E\to T$, which  corresponds to a class
in $\Ext^2(T,F)$. 

The torsion theory $(\kt,\kf)$ for $\ka$ gives rise to the torsion theo\-ry
$(\kf[1],\kt)$ for the tilt $\ka^\sharp$. If $\kt$ and $\kf$ are both non-trivial, then $\ka\ne\ka^\sharp$
(even after shift). If $\kf=0$, then $\ka=\ka^\sharp$
and if $\kt=0$, then $\ka^\sharp=\ka[1]$.

The important fact is the following result proved in \cite{HRS}, which we leave as an exercise (use
Remark \ref{rem:BrAnn}).

\begin{prop}
The category $\ka^\sharp$ is the heart of a bounded t-structure on $\Db(\ka)$.\qqed
\end{prop}
The analogous statement for the heart $\ka$ of a bounded t-structure on  a triangulated category
$\kd$ also holds. 

\begin{remark}
At this point it is natural to wonder what the relation between $\Db(\ka)$ and $\Db(\ka^\sharp)$ is. This
question is addressed in \cite{HRS} and \cite{BvdB}. E.g.\ it is proved that the bounded derived categories
of $\ka$ and its tilt $\ka^\sharp$ are equivalent, if the torsion theory $(\kt,\kf)$ for $\ka$ is \emph{cotilting}.
This means that for all $E\in\ka$ there exists a `torsion free' object $F\in\kf$ and an epimorphism $F\twoheadrightarrow E$.
Similarly, a torsion theory is \emph{tilting} if for all $E\in\ka$ there exists a `torsion' object $T\in\kt$ and a monomorphism
$E\,\,\hookrightarrow T$. One verifies that $(\kt,\kf)$ is tilting for $\ka$ if and only if $(\kf[1],\kt)$ is cotilting for
$\ka^\sharp$ (see \cite{HRS}).
\end{remark}

Our standard example (see Remark \ref{rem:torsionCoh(C)}) starts with the torsion theory $(\kt_\varphi,\kf_\varphi)$, $\varphi\in(0,1]$,
for $\coh(C)$, where as before $C$ is a smooth projective curve. We call its tilt $\ka_\varphi\subset\Db(C)$.
This construction describes a \emph{family of bounded t-structures} that depends on the parameter $\varphi\in(0,1]$. Later however, we will see that this family is induced by a natural $\widetilde{\rm GL}\!{}^+(2,\RR)$-action on the space of stability conditions. But
there are more interesting families of t-structures that are not of this form.

The following picture of the image of $\ka=\coh(C)$ and $\ka_\varphi$ under the stability function
$Z=-\deg+i\cdot\rk$ might be helpful to visualize the tilt. In particular, for a $\mu$-stable vector bundle $E$ of slope
$\mu(E)$ one has $E\in\kt_\varphi\subset\ka_\varphi$ if $\mu(E)<-\cot(\pi\varphi)$, but
$E[1]\in\kf_\varphi\subset\ka_\varphi$ if $\mu(E)\geq-\cot(\pi\varphi)$.

$$\hskip-9cm
\begin{picture}(-100,100)
{  \linethickness{0.055mm}
\put(-80,20){\line(1,0){80}}
\put(2,20){\circle{4}}
\put(0.7,22){\line(-1,1){50}}
}
\put(-44,40){\tiny\mbox{$\kt$}}
\put(30,50){\tiny\mbox{$\kf$}}
{  \linethickness{0.015mm}
\put(4,20){\line(1,0){10}}
\put(24,20){\line(1,0){10}}
\put(44,20){\line(1,0){10}}
\put(-34,20){\line(0,1){34}}
\put(-44,20){\line(0,1){44}}
\put(-54,20){\line(0,1){50}}
\put(-64,20){\line(0,1){50}}
\put(-24,20){\line(0,1){24}}
\put(-14,20){\line(0,1){14}}
\put(-7.5,30){\line(1,0){44}}
\put(-17.5,40){\line(1,0){44}}
\put(-27.5,50){\line(1,0){44}}
\put(-37.5,60){\line(1,0){44}}
\put(-47.5,70){\line(1,0){44}}

}
\put(67,30){\mbox{$\leadsto$}}

{  \linethickness{0.055mm}
\put(101,20){\line(1,0){80}}
\put(183,20){\circle{4}}}
{ \linethickness{5mm}
\put(184.7,18){\line(1,-1){40}}
}
\put(136,40){\tiny\mbox{$\kt$}}
\put(145,-10){\tiny\mbox{$\kf[1]$}}
{  \linethickness{0.015mm}
\put(181,22){\line(-1,1){10}}
\put(168,35){\line(-1,1){10}}
\put(155,48){\line(-1,1){10}}
\put(142,61){\line(-1,1){10}}
\put(124,20){\line(0,1){45}}
\put(134,20){\line(0,1){45}}
\put(144,20){\line(0,1){35}}
\put(154,20){\line(0,1){25}}
\put(164,20){\line(0,1){15}}
\put(222.5,-20){\line(-1,0){44}}
\put(212.5,-10){\line(-1,0){44}}
\put(202.5,0){\line(-1,0){44}}
\put(192.5,10){\line(-1,0){44}}
}
\end{picture}
$$
\smallskip
\vskip1cm

\begin{remark} In \cite{Pol} it is shown that for the various
$\varphi$ the tilts $\ka_\varphi\subset\Db(C)$ for an elliptic curve
$C$ describe the categories of holomorphic vector bundles on non-commutative deformations of $C$.
Roughly, for $C=\CC/(\ZZ+\tau\ZZ)$ and $\theta\in\RR\setminus\QQ$ one considers the algebra $A_\theta$ of series
$\sum a_{m,n}U_1^m\cdot U_2^n$ with rapidly decreasing coefficients and the rule $U_1\cdot U_2=\exp(2\pi i\theta)
U_2\cdot U_1$. A holomorphic vector bundle on the non-commutative $2$-torus $T_{\tau,\theta}$ is 
by definition  an $A_\theta$-module $E$ with a `connection' $\nabla:E\to E$ satisfying the Leibniz rule with respect
to the derivative $\delta$ given by $\delta(U_1^m\cdot U_2^n)=(2\pi i)(m\tau+n)U_1^m\cdot U_2^n$. Then for $\varphi={\rm arcot}(-\theta)$ the category of holomorphic vector bundles on $T_{\tau,\theta}$ is equivalent to $\ka_\varphi$.
\end{remark}

The relation between tilts of  the heart $\ka$ of a t-structure on a triangulated category $\kd$ and new t-structures 
on $\kd$ is clarified by the following result, see \cite[Lem.\ 1.1.2]{Polt}, \cite[Thm.\ 3.1]{BR} or \cite[Prop.\ 2.3]{Wolf1}.

\begin{prop}\label{prop:tiltssand}
If $\ka=\kd^{\leq0}\cap\kd^{\geq0}$ denotes the heart of a t-structure on a triangulated category
$\kd$ given by $\kd^{\leq 0}\subset\kd$. Then there exists a natural bijection between:

i) Torsion theories for $\ka$ and

ii) t-structures on $\kd$ given by $\kd'^{\leq0}\subset\kd$ satisfying $\kd^{\leq0}[1]\subset \kd'^{\leq 0}\subset\kd^{\leq0}$.
\end{prop}

\begin{proof}
We only define the maps: A torsion theory $(\kt,\kf)$ for $\ka$ yields a t-structure 
$\kd'^{\leq0}:=\{E~|~H^i_\ka(E)=0~i>0 \text{~and~}H^0_\ka(E)\in\kt\}$. Here, $H^i_\ka$ denotes the cohomology
with respect to the original t-structure which, therefore, takes values in $\ka$. Conversely, if $\kd'^{\leq0}$ is given,
then one defines a torsion theory by $\kt:=\ka\cap\kd'^{\leq0}$ which can also be written
as $\{E\in\ka~|~E\in H^0_{\ka'}(\kd'^{\leq0})\}$. 
\end{proof}

So roughly, one can tilt until $\kd'^{\leq0}$ leaves the slice $\kd^{\leq-1}\subset\kd^{\leq0}$. The explicit description of torsion theories versus t-structures is later used often in the analysis of the
space of stability conditions.

\section{Stability conditions: Definition and examples}

In this second lecture, slicings and stability conditions are defined on arbitrary triangulated categories.
Section \ref{sect:GL} discusses the natural action of $\widetilde{\rm GL}\!{}^+(2,\RR)$ 
and $\Aut(\kd)$ on the space $\Stab(\kd)$ of stability conditions and
in Section \ref{sect:stabcurves} it is shown that the $\widetilde{\rm GL}\!{}^+(2,\RR)$-action is transitive when $\kd$ is the bounded derived category
$\Db(C)$ of a curve of genus $g(C)>0$. 
\subsection{Slicings} We shall introduce this notion simultaneously for abelian categories $\ka$
and triangulated categories $\kd$. 
It is maybe easier to grasp for abelian categories and a slicing on a triangulated category is in fact nothing but a
bounded t-structure
with a slicing of the abelian category given by its heart (cf.\ Proposition \ref{prop:boundedslicing}). The analogous notions of a torsion theory and a t-structure are both refined by the notion of a
slicing which formalizes the example of the categories $\kp(\phi)\subset \coh(C)$ of semistable sheaves of phase $\phi$ on a curve
(see Definition \ref{dfb:catsstPcurve}).

\begin{definition}\label{def:slic}
A \emph{slicing} $\kp=\{\kp(\phi)\}$ of $\ka$ (resp.\ $\kd$) consists of full additive (not necessarily abelian)
subcategories $\kp(\phi)\subset\ka$, $\phi\in(0,1]$ (resp.\ $\kp(\phi)\subset\kd$, $\phi\in\RR$) such that:

i) $\Hom(\kp(\phi_1),\kp(\phi_2))=0$ if $\phi_1>\phi_2$.

ii) For all $0\ne E\in\ka$ there exists a filtration $0=E_0\subset E_1\subset\ldots\subset E_n=E$ with
$A_i:=E_i/E_{i-1}\in\kp(\phi_i)$ and $\phi_1>\ldots>\phi_n$.
(For $E\in\kd$ the inclusions $E_{i-1}\subset E_i$ are replaced by morphisms
$E_{i-1}\to E_i$ and the quotients $A_i$ are defined as their cones.)

$$
\hskip-3cm
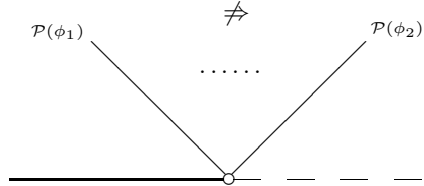
\begin{picture}(-100,100)
{  \linethickness{0.25mm}
\put(-80,20){\line(1,0){80}}

\put(-9,60){\small\mbox{\ldots\ldots}}

\put(2,20){\circle{4}}
\put(0.7,22){\line(-1,1){50}}
\put(3.3,22){\line(1,1){50}}
}
\put(-72,74){\tiny\mbox{$\kp(\phi_1)$}}
\put(55,75){\tiny\mbox{$\kp(\phi_2)$}}
{  \linethickness{0.015mm}
\put(4,20){\line(1,0){10}}
\put(24,20){\line(1,0){10}}
\put(44,20){\line(1,0){10}}
\put(64,20){\line(1,0){10}}
}
\put(0,80){\mbox{$\not\Rightarrow$}}

\end{picture}
$$

\smallskip

iii) This condition is for the triangulated case only:  $\kp(\phi)[1]=\kp(\phi+1)$ for all $\phi\in\RR$.
\end{definition}

As before, we will call $\kp(\phi)$ the subcategory of \emph{semistable} objects of phase $\phi$. 
The minimal objects in $\kp(\phi)$, i.e.\ those without any proper subobject in $\kp(\phi)$, are called
\emph{stable}
of phase $\phi$. The filtration in ii) is the
Harder--Narasimhan filtration, which is again unique due to i), and the $A_i$ are the
semistable factors of $E$. Also, for any interval
$I$ we let $\kp(I)$  be the full subcategory of objects $E$, for which the semistable factors have phases in $I$.
We repeat the warning that the phase $\phi(E)$ of an object is only well defined for semistable $E$.

\begin{remark}\label{rem:useful}
Here are a few useful observations:

i) If $E\to B$ is a non-trivial morphism, then for at least one semistable factor $A_i$ of $E$ there exists
a non-trivial morphism $A_i\to B$.

ii) If $A\in\kp(\phi)$, $B\in\kp(I)$, and $\phi>t$ for all $t\in I$, then $\Hom(A,B)=0$.
	
iii) For the Harder--Narasimhan filtration of $E\in\kd$ all the morphisms $E_i\to E$ are non-trivial and
similarly the projection $E\to A_n$ is non-trivial.

iv) If $E$ is not semistable, i.e.\ not contained in any of the $\kp(\phi)$, then there exists a short exact sequence
$0\to A\to E\to B\to 0$ (resp.\ an exact triangle $A\to E\to B\to A[1]$) with $A,B\ne0$ and $\Hom(A,B)=0$.	
\end{remark}

\begin{ex} 
Consider a torsion theory $(\kt,\kf)$ for an abelian category $\ka$. Pick $1\geq \phi_1>\phi_2>0$ and
let $\kp(\phi_1):=\kt$, $\kp(\phi_2):=\kf$, and $\kp(\phi)=0$ for $\phi\ne\phi_1,\phi_2$. This then defines a slicing of $\ka$. 
In particular, it provides examples of slicings with  slices $\kp(\phi)$ which are not(!) abelian.

 Similarly, if a t-structure on a triangulated category $\kd$ has heart $\ka$, then
for any choice of $\phi_0\in\RR$ one can define a slicing of $\kd$ by $\kp(\phi):=\ka[\phi-\phi_0]$ if $\phi-\phi_0\in\ZZ$
and $\kp(\phi)=0$ otherwise.
\end{ex}

Suppose  $\kp=\{\kp(\phi)\}$ is a slicing of an abelian category $\ka$. Pick  $\phi_0\in(0,1]$. Then  
$\kt:=\kp(>\!\phi_0)$ and $\kf:=\kp(\leq\!\phi_0)$  define a
torsion theory for $\ka$. In the triangulated setting this yields the following
assertion where we choose $\phi_0=0$.	
	
\begin{prop}\label{prop:boundedslicing}
Let $\kd$ be a triangulated category. There is natural bijection between

i) Slicings of $\kd$ and

ii) Bounded t-structures on $\kd$ together with a slicing of their heart.
\end{prop}

\begin{proof} Start with a slicing $\{\kp(\phi)\}_{\phi\in\RR}$ of $\kd$.
Then define  a bounded t-structure on $\kd$ by $\kd^{\leq 0}=\kp(>0)$.\footnote{This clash of
notation really is unavoidable.} Its heart is $\ka=\kp(0,1]$ and $\{\kp(\phi)\}_{\phi\in(0,1]}$ defines a slicing of it. 
Conversely, if a slicing $\{\kp(\phi)\}_{\phi\in(0,1]}$ of the heart $\ka$ of a bounded t-structure
on $\kd$ is given, define
a slicing of $\kd$ by $\kp(\phi+k):=\kp(\phi)[k]$, $k\in\ZZ$.
\end{proof}
	
\subsection{Stability conditions}
To make slicings a really useful notion, they have to be linearized by a stability function, which is  the abstract version
of  $Z=-\deg+i\cdot\rk$ on $\Db(C)$. This will lead to the concept of stability conditions. The categories $\kp(\phi)$ of semistable objects
will then automatically be abelian. 

There are two possible approaches to stability conditions. Either one starts with a slicing and searches for a compatible stability function or a natural stability function is given already and one verifies that it also defines a slicing, i.e.\ that 
Harder--Narasimhan filtrations exist.

Once more, we will first deal with the easier notion of stability conditions on abelian categories.
So let $\ka$ be an abelian category. As before, $K(\ka)$ will denote its Grothendieck group,
i.e.\ the abelian group generated by objects of $\ka$ subject to the relation $[E_2]=[E_1]+[E_3]$
for any short exact sequence $0\to E_1\to E_2\to E_3\to0$. 

\emph{Warning:} Despite its easy definition, it is usually very difficult
to control the Grothendieck group $K(\ka)$.

\begin{definition}
A \emph{stability function} on an abelian category $\ka$ is a linear map $Z:K(\ka)\to \CC$ such that $Z(E)=Z([E])\in
\overline\HH=\HH\cup\RR_{<0}$ for any $0\ne E\in\ka$.
\end{definition}

In particular, with respect to a given stability function $Z$ any non-trivial object $E\in\ka$ has a phase  $\phi(E)\in(0,1]$
which is determined by $Z(E)\in\exp(i\pi\phi(E))\cdot\RR_{>0}$.

Moreover, for a short exact sequence $0\to E_1\to E_2\to E_3\to0$ the phases $\phi_i$ of $E_i$, $i=1,2,3$
are related by $Z(E_2)=Z(E_1)+Z(E_3)$ as in:

$$
\hskip-3cm
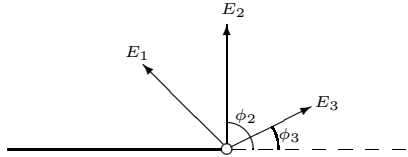
\begin{picture}(-100,100)
{\linethickness{0.25mm}
\put(-80,20){\line(1,0){80}}}
\put(2,20){\circle{4}}
\put(0.7,22){\vector(-1,1){30}}
{\linethickness{0.0025mm}
\put(3.6,21){\vector(2,1){30}}
\put(2,22.2){\vector(0,1){45}}
\put(2,20){\oval(20,20)[rt]}
{ \linethickness{0.15mm}\put(4,20){\qbezier(17.4,0)(17.3,5.4)(14.9,8.5)}}
\put(0,71){\tiny\mbox{$E_2$}}
\put(5,31){\tiny\mbox{$\phi_2$}}
\put(-36,55){\tiny\mbox{$E_1$}}
\put(35,36){\tiny\mbox{$E_3$}}
\put(21.4,24){\tiny\mbox{$\phi_3$}}
}
{  \linethickness{0.015mm}
\put(4,20){\line(1,0){5}}
\put(14,20){\line(1,0){5}}
\put(24,20){\line(1,0){5}}
\put(34,20){\line(1,0){5}}
\put(44,20){\line(1,0){5}}
\put(54,20){\line(1,0){5}}
\put(64,20){\line(1,0){5}}
}
\end{picture}
$$

As in the case of $\coh(C)$ discussed in Section \ref{sec:Recol}, an object $E\in\ka$ is called \emph{semistable} with respect to a given stability function  $Z$ if it does not contain any subobject of phase
$>\!\phi(E)$. Then, one defines the subcategory $\kp(\phi)$
of semistable objects (with respect to $Z$) of phase $\phi$ as in Definition \ref{dfb:catsstPcurve}.
A stability function is said to satisfy the \emph{Harder--Narasimhan property} if  any object has a finite filtration with semistable quotients $A_i$, i.e.\ if $\{\kp(\phi)\}_{\phi\in(0,1]}$ is a slicing of $\ka$.

Conversely, if a slicing $\{\kp(\phi)\}_{\phi\in(0,1]}$ is given, then a stability function $Z$
is \emph{compatible} with it  if $Z(E)\in\exp(i\pi\phi)\cdot\RR_{>0}$ for all $0\ne E\in\kp(\phi)$.

\begin{definition}
A \emph{stability condition} on an abelian category $\ka$ is  a pair $\sigma=(\kp,Z)$ consisting 
of a slicing $\kp$  and a compatible stability function $Z$.\footnote{Later, a finiteness condition will be added, see Definition
\ref{dfn:locfinite}, which in particular implies that any semistable object  has a finite Jordan--H\"older filtration.}
\end{definition}

\begin{remark}
The datum of a stability condition is thus equivalent to the datum of a stability function $Z:K(\ka)\to \CC$ satisfying the Harder--Narasimhan property. 
\end{remark}

\begin{cor}
If $\sigma=(\kp,Z)$ is a stability condition, then all categories
$\kp(\phi)$ are abelian.
\end{cor}

\begin{proof} See \cite[Lem.\ 5.2]{BrAnn}.
In order to make $\kp(\phi)$ abelian, one has to show that kernels, images, etc., exist. Since $\kp(\phi)$ is a
subcategory of the given abelian category $\ka$, one can consider  a morphism $f:E\to F$ in $\kp(\phi)$
as a morphism in $\ka$ and take its kernel there. In order to conclude, one has to show that $\ker(f)$ is in fact
an object in $\kp(\phi)$ (and similarly for image, cokernels, etc.). Suppose first that
$\ker(f)$ and ${\rm Im}(f)$  are semistable, i.e.\ $\ker(f)\in\kp(\psi)$ and ${\rm Im}(f)\in\kp(\psi')$  for certain $\psi,\psi'$.
Since $\ker(f)\subset E$ and ${\rm Im}(f)\subset F$, semistability of $E$ resp.\ $F$ yields $\psi\leq \phi$
and $\psi'\leq\phi$. Thus, the images of $\ker(f)$ and ${\rm Im}(f)$ under $Z$ are  to the right of 
the ray $Z(E)\cdot\RR_{>0}= Z(F)\cdot\RR_{>0}$ as pictured below.

$$
\hskip-3cm
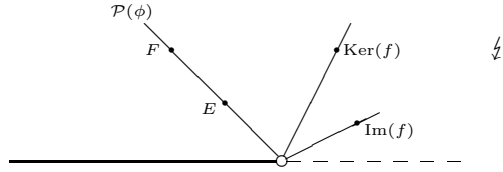
\begin{picture}(-100,100)
{\linethickness{0.25mm}
\put(-100,20){\line(1,0){100}}}
\put(2,20){\circle{4}}
\put(0.7,22){\line(-1,1){50}}
\put(-28,38){\tiny\mbox{$E$}}
\put(-19.2,42){\circle*{2}}
\put(-49,60){\tiny\mbox{$F$}}
\put(-39.2,62){\circle*{2}}
\put(-62,75){\tiny\mbox{$\kp(\phi)$}}
{\linethickness{0.0025mm}
\put(3.6,21){\line(2,1){35}}
\put(2.9,22.2){\line(1,2){25}}
\put(33,30){\tiny\mbox{${\rm Im}(f)$}}
\put(30.2,34.3){\circle*{2}}
\put(25,60){\tiny\mbox{${\rm Ker}(f)$}}
\put(22.6,62){\circle*{2}}

}
{  \linethickness{0.015mm}
\put(4,20){\line(1,0){5}}
\put(14,20){\line(1,0){5}}
\put(24,20){\line(1,0){5}}
\put(34,20){\line(1,0){5}}
\put(44,20){\line(1,0){5}}
\put(54,20){\line(1,0){5}}
\put(64,20){\line(1,0){5}}
}
\put(80,60){\mbox{$\lightning$}}
\end{picture}
$$

\smallskip

But this contradicts the linearity of $Z$ which says $Z(E)=Z(\ker(f))+Z({\rm Im}(f))$.
For the general case, one uses the Harder--Narasimhan filtration of $\ker(f)$ and ${\rm Im}(f)$ and again the linearity
of $Z$ to reduce to the case above.
\end{proof}

Let us now pass to the case of a triangulated category $\kd$. Its Grothendieck group
$K(\kd)$ is defined to be the abelian group generated by objects $E$ in $\kd$ subject to the 
obvious relation
defined by exact triangles. Note that with this definition one has 
 $K(\kd)\cong K(\ka)$, whenever $\ka$ is the heart of bounded t-structure on $\kd$. 
 
 \emph{Warning:} For triangulated
 subcategories $\kd'\subset\kd$ the induced map $K(\kd')\to K(\kd)$ is in general not injective.\footnote{For dense subcategories it is, see \cite{TT}.}
 
 \begin{definition}
 A \emph{stability function} on the triangulated category $\kd$ is a linear map
 $Z:K(\kd)\to\CC$.
 \end{definition}
Note that we do not make any assumption on the image of $Z$. Indeed, the linearity of $Z$ in particular implies
 that $Z(E[1])=-Z(E)$, as $[E[1]]=-[E]$ in $K(\kd)$. Also note that there are always non-trivial objects in the kernel of $Z$. Indeed
  $Z(E\oplus E[1])=0$ for all $E$.
 
As before, we say that $Z$ is compatible with a slicing $\{\kp(\phi)\}_{\phi\in\RR}$ of $\kd$ if $Z(E)\in\exp(i\pi\phi)\cdot
\RR_{>0}$
for all $0\ne E\in\kp(\phi)$. 

\emph{Warning:} In contrast to the abelian case, the phase of an arbitrary object
$E\in\kd$ is not well defined (not even modulo $\ZZ$).

\begin{definition}
A \emph{stability condition} on a triangulated category $\kd$  is a pair $\sigma=(\kp,Z)$
consisting of a slicing $\kp=\{\kp(\phi)\}_{\phi\in\RR}$ of $\kd$
and a compa\-tible stability function $Z:K(\kd)\to \CC$.\footnote{Later, a finiteness condition will again be added, see Definition
\ref{dfn:locfinite}.}
\end{definition}

In particular, the \emph{space of all stability conditions} $${\rm Stab}(\kd):=\{\sigma=(\kp,Z)\}$$ comes
with a natural map $$\pi:{\rm Stab}(\kd)\to K(\kd)^*:=\Hom(K(\kd),\CC), ~~\sigma=(\kp,Z)\mapsto Z.$$

\begin{ex}\label{ex:patho}
Here is a pathological example (cf.\ \cite[Ex.\ 5.6]{BrAnn}) that we will want to avoid by adding local finiteness
(see Definition \ref{dfn:locfinite}) later on.
Consider a curve $C$, $\varphi\in\RR\setminus\QQ$ and the induced torsion theory as in Section \ref{sec:tilts}
with its tilt $\ka_\varphi\subset\Db(C)$. Define $Z:K(\kd)=K(\ka_\varphi)\to \CC$ as $Z(E)=i\cdot(\deg(E)-\varphi\cdot\rk(E))$.
Then $Z(E)\in i\cdot\RR_{>0}$ for all $0\ne E\in\ka_\varphi$. In particular, all objects of $\ka_\varphi$ are semistable
of phase $\phi=\frac{1}{2}$, i.e.\ $\kp(\frac{1}{2})=\ka_\varphi$ and $\kp(\phi)=\{0\}$ for $\phi\ne\frac{1}{2}$.
In this example, semistable objects do not necessarily admit finite Jordan--H\"older filtrations. In fact, for  any
stable vector bundle $E$ of slope $<\varphi$ one obtains a strictly descending filtration
$\ldots E_i[1]\subset E_{i-1}[1]\subset\ldots E_0[1]=E[1]$ in $\kp(\frac{1}{2})$,
where $E_i$ is defined as the kernel of some arbitrarily chosen
surjection $E_{i-1}\twoheadrightarrow k(x_i)$.
\end{ex}
Building upon Proposition \ref{prop:boundedslicing}, one obtains (cf.\ \cite[Prop.\ 5.3]{BrAnn}):
\begin{prop} Let  $\kd$ be a triangulated category. There is natural bijection between

i) Stability conditions on $\kd$ and

ii) Bounded t-structures on $\kd$ together with a stability condition on  the heart.\qqed
\end{prop}

By $\ka_\sigma$ we shall denote the heart of a stability condition $\sigma$, so $\ka_\sigma=\kp(0,1]$.

It may be instructive to  picture the image of the hearts of the various t-structures associated with the slicing underlying a  stability function on $\kd$.
\vskip-1cm
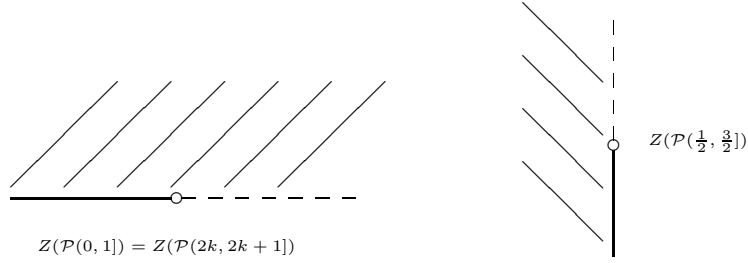
$$\hskip-9cm
\begin{picture}(-100,100)
 {\linethickness{0.005mm}
\put(-60,24){\line(1,1){40}}
\put(-40,24){\line(1,1){40}}
\put(-20,24){\line(1,1){40}}
\put(0,24){\line(1,1){40}}
\put(20,24){\line(1,1){40}}
\put(40,24){\line(1,1){40}}
 }
{  \linethickness{0.25mm}
\put(-60,20){\line(1,0){60}}
\put(2,20){\circle{4}}}
{  \linethickness{0.015mm}
\put(4,20){\line(1,0){5}}
\put(14,20){\line(1,0){5}}
\put(24,20){\line(1,0){5}}
\put(34,20){\line(1,0){5}}
\put(44,20){\line(1,0){5}}
\put(54,20){\line(1,0){5}}
\put(64,20){\line(1,0){5}}
}
\put(-50,0){\tiny\mbox{$Z(\kp(0,1])=Z(\kp(2k,2k+1])$}}
\end{picture}
$$
\vskip-4.8cm
$$\hskip2.5cm
\begin{picture}(-100,100)
 {\linethickness{0.005mm}
\put(-2,44){\line(-1,1){30}}
\put(-2,24){\line(-1,1){30}}
\put(-2,4){\line(-1,1){30}}
\put(-2,-16){\line(-1,1){30}}

 }
{  \linethickness{0.25mm}
\put(2,-22){\line(0,1){40}}
\put(2,20){\circle{4}}}
{  \linethickness{0.015mm}
\put(2,22){\line(0,1){5}}
\put(2,32){\line(0,1){5}}
\put(2,42){\line(0,1){5}}
\put(2,52){\line(0,1){5}}
\put(2,62){\line(0,1){5}}
}
\put(15,20){\small\mbox{\tiny$Z(\kp(\frac{1}{2},\frac{3}{2}])$}}
\end{picture}
$$
\vskip2cm

\begin{remark}
To avoid the often mysterious $K(\kd)$, one usually fixes a finite rank quotient $K(\kd)\twoheadrightarrow \Gamma\cong\ZZ^d$
and requires that $Z$ factors via $\Gamma$, i.e.\ $Z:K(\kd)\twoheadrightarrow\Gamma\to\CC$. A typical example
is that of a $k$-linear triangulated category $\kd$ with finite-dimensional $\bigoplus_i\Hom(E,F[i])$ for all $E,F\in\kd$.
Then one can consider the \emph{numerical Grothendieck group}
 $$N(\kd)=\Gamma=K(\kd)/_\sim,$$ where $E\sim 0$ if $\chi(E,F):=\sum(-1)^i\dim\Hom(E,F[i])=0$ for
all $F$.  Stability conditions of this type will be called \emph{numerical}.

In the geometric situation, i.e.\ when  $\kd=\Db(X)$ with $X$  a smooth complex projective variety,
the numerical Grothendieck group $N(\Db(X))\otimes \QQ$ should
be, according to one of the standard conjectures,
isomorphic to the algebraic cohomology $H^*_{\rm alg}(X,\QQ)$ (and according to the Hodge conjecture in fact to
$\bigoplus H^{p,p}(X)\cap H^*(X,\QQ)$).
\end{remark}
\subsection{$\Aut(\kd)$-action and $\widetilde{\rm GL}\!{}^+(2,\RR)$-action}\label{sect:GL} We next come to a feature that is not present in the case of an abelian category.
Let us start with the natural right action 
$$\CC\times {\rm GL}(2,\RR)\to \CC$$
given by $(z,M)\mapsto M^{-1}\cdot z$ via the natural identification $\CC\cong\RR^2$. This yields
$$K(\kd)^*\times{\rm GL}(2,\RR)\to K(\kd)^*, ~~(Z,M)\mapsto M^{-1}\circ Z.$$

Since the image of the heart $\kp(0,1]$ of a stability condition on $\kd$ is contained in $\overline\HH$, it might be helpful
to picture the image of $\overline\HH$ under some standard elements of ${\rm GL}(2,\RR)$. E.g.\ $-{\rm id}=\left(\begin{matrix}-1&0\\
0&-1\end{matrix}\right)$ acts as rotation by $\pi$ (anti-clockwise)
%
%
and $-i=\left(\begin{matrix}0&-1\\1&0\end{matrix}\right)$ as rotation by $\frac{\pi}{2}$ clockwise.
%

 We would like to lift the action of ${\rm GL}(2,\RR)$ on $K(\kd)^*$ to an action on ${\rm Stab}(\kd)$ 
under the natural projection $\pi:{\rm Stab}(\kd)\to K(\kd)^*$. Due to the phases of semistable
objects in $\kd$ being contained in $\RR$ and not only in $(0,1]$, one has to pass
to the universal cover of ${\rm GL}(2,\RR)$ first. In fact, the same problem occurs when 
one wants to lift the action $\CC\times {\rm GL}(2,\RR)\to \CC$ with respect to the exponential map $\exp:\CC\to\CC$.
For stability conditions there is
one more aspect to be taken into account. As one wants the property $\Hom(\kp(\phi_1),\kp(\phi_2))=0$
for $\phi_1>\phi_2$, to be preserved under the action, one needs to restrict to the connected component
${\rm GL}\!{}^+(2,\RR)$ of invertible matrices with positive determinant.

\begin{lem} The universal cover of ${\rm GL}\!{}^+(2,\RR)$ can  be described explicitly
as the group $$\widetilde{\rm GL}\!{}^+(2,\RR)=\{(M,f)~|~M\in{\rm GL}\!{}^+(2,\RR),~ f:\RR\to \RR,~i),~ii) \}$$
with:

 i) $f$ is increasing with $f(\phi+1)=f(\phi)+1$ for all $\phi\in\RR$ and

ii) $M\cdot \exp(i\pi\phi)\in \exp(i\pi f(\phi))\cdot\RR_{>0}$.
\end{lem}

Condition ii) simply says that 
the induced maps $\bar M$ and $\bar f$ on $(\RR^2\setminus\{0\})/\RR_{>0}=S^1=\RR/2\ZZ$ coincide.

Then the action of ${\rm GL}\!{}^+(2,\RR)$ on $\CC$ is lifted naturally to
$$\xymatrix{{\rm Stab}(\kd)\times\widetilde{\rm GL}\!{}^+(2,\RR)\ar[r]&{\rm Stab}(\kd)}$$
\vskip-0.5cm
$$\xymatrix{(\sigma=(\kp,Z),(M,f))\ar@{|->}[r]&\sigma'=(\kp',Z')}$$
with $Z'=M^{-1}\circ Z$ and $\kp'(\phi)=\kp(f(\phi))$.

For example, ${\rm id}$ can be lifted to $({\rm id}, f:\mapsto\phi+2k)\in\widetilde{\rm GL}\!{}^+(2,\RR)$
with arbitrary $k\in\ZZ$. It acts on ${\rm Stab}(\kd)$ by fixing the stability function and by changing a given slicing
$\{\kp(\phi)\}_{\phi\in\RR}$ to $\{\kp'(\phi)=\kp(\phi+2k)\}_{\phi\in\RR}$. Thus, in terms of its action on ${\rm Stab}(\kd)$
the universal cover of ${\rm GL}\!{}^+(2,\RR)$
can be seen as an extension
$$\xymatrix{0\ar[r]&\ZZ[2]\ar[r]&\widetilde{\rm GL}\!{}^+(2,\RR)\ar[r]&{\rm GL}\!{}^+(2,\RR)\ar[r]&0.}$$

Similarly,
$-{\rm id}$  can be lifted to $(-{\rm id},f:\phi\mapsto\phi+2k+1)\in\widetilde{\rm GL}\!{}^+(2,\RR)$
with arbitrary $k\in\ZZ$. It  sends a slicing $\kp$ to $\{\kp'(\phi)=\kp(\phi+2k+1)\}$.

If one thinks of a stability condition as a way of decomposing the triangulated category in small slices, then stability conditions
in the same $\widetilde{\rm GL}\!{}^+(2,\RR)$-orbit decompose the category essentially
in the same way. In particular, the action rotates the slices around without changing the stability of an object.
Compare the comments in Section \ref{sec:tilts}. Thus, if one wants to really understand the different possibilities
for stability conditions on $\kd$ it is rather the quotient $$\raisebox{1mm}{${\rm Stab}(\kd)$}/
\raisebox{-1mm}{$\widetilde{\rm GL}\!{}^+(2,\RR)$}$$
one needs to describe.

Any linear exact autoequivalence $\Phi$ of $\kd$, i.e.\ $\Phi\in\Aut(\kd)$, induces an action on the
Grothendieck group  $K(\kd)$ by $[E]\mapsto [\Phi(E)]$ which we shall also denote $\Phi$.
Then $\Aut(\kd)$ acts on ${\rm Stab}(\kd)$  via
$$\xymatrix{\Aut(\kd)\times{\rm Stab}(\kd)\ar[r]&{\rm Stab}(\kd)}$$
\vskip-0.6cm
$$\xymatrix{(\Phi,\sigma=(\kp,Z))\ar@{|->}[r]&(\kp',Z\circ\Phi^{-1}),}$$
with $\kp'(\phi)=\Phi(\kp(\phi))$. The left $\Aut(\kd)$-action
and the right $\widetilde{\rm GL}\!{}^+(2,\RR)$-action obviously commute. 
\subsection{Stability conditions on curves}\label{sect:stabcurves}
We come back to the derived category $\Db(C)$ of complexes of coherent sheaves on a smooth projective
curve over an algebraically closed field $k$. It is possible to describe the space of numeri\-cal stability condition
completely, at least for $g(C)>0$.\footnote{There are, of course, easier examples one could look at. E.g.\ stability conditions on the abelian category of finite dimensional vector spaces and on its bounded derived category, 
the abelian category of Hodge structures, etc.
}

We have tried to motivate the abstract  notion of a stability condition by the example on $\Db(C)$ with
$Z=-\deg+i\cdot\rk$ and the standard bounded t-structure on $\Db(C)$ with heart $\coh(C)$. The Harder--Narasimhan proper\-ty  holds for $Z$ and the additional finiteness condition we have alluded to  is essentially the existence of finite Jordan--H\"older filtration. Both are ultimately  
explained by the rationality of $-\deg+i\cdot \rk$ (see Remark \ref{rem:discZ}).

The following result,  due to Bridgeland \cite{BrAnn} and Macr\`i \cite{Ma}, shows that 
up to the action of $\widetilde{\rm GL}\!{}^+(2,\RR)$ there is only one stability condition on $\Db(X)$, namely 
the one with the classical choice $Z=-\deg+i\cdot \rk$ and the standard t-structure on $\Db(C)$.

\begin{thm} The space of numerical stability conditions ${\rm Stab}(C):={\rm Stab}(\Db(C))$ on $\Db(C)=\Db(\coh(C))$
of a smooth projective curve
$C$ of genus $g(C)>0$ over a field $k=\bar k$ consists of exactly one $\widetilde{\rm GL}\!{}^+(2,\RR)$-orbit:
$$\raisebox{1mm}{${\rm Stab}(C)$}/\raisebox{-1mm}{$\widetilde{\rm GL}\!{}^+(2,\RR)$}=\{{\rm pt}\}.$$
\end{thm}

\begin{proof}
First, recall the following fact that works for arbitrary abelian cate\-gories of homological dimension $\leq1$: Any
object $E\in\Db(C)$ can be written as $E\cong\bigoplus E_i[-i]$ with $E_i\in\coh(C)$ (see e.g.\ \cite[Cor.\ 3.15]{HuyFM}).
This is proved by induction on the length of a complex. If $E$ is concentrated in degree $\geq i$, then
there exists an exact triangle $H^i(E)[-i]\to E\to F$ with $H^j(F)=H^j(E)$ for $j>i$ and $=0$ otherwise.
So, we can assume $F\cong\bigoplus F_j[-j]$.
But then the boundary map is in $\Ext^1(F,H^i(E)[-i])=\bigoplus_{j>i} \Ext^{1+j-i}(F_j,H^i(E))$ which is trivial,
as $1+j-i>1$.

The key Lemma \ref{lem:keycurves} below says that all point sheaves
$k(x)$ and all line bundles $L\in\Pic(C)$ are stable with respect to any $\sigma\in{\rm Stab}(C)$. Using this,
one concludes as follows. Let $\phi_x$ be the phase of the stable object $k(x)$ and let $\phi_L$ be the phase of the stable object $L\in\Pic(C)$. Since $\Hom(L,k(x))\ne0$ and by Serre duality also $\Ext^1(k(x),L)\cong\Hom(L,k(x))^*\ne0$, 
stability yields (see Exercise \ref{rem:GP}) $\phi_L<\phi_x$ and $\phi_x-1<\phi_L$.

The numerical Grothendieck group of $C$ is $N(C)\cong\ZZ\oplus\ZZ$ generated by rank and degree.
E.g.\ $[k(x)]=(0,1)$ and $[L]=(1,\deg(L))$. Thus, $Z:N(C)\otimes\RR\cong\RR^2\to \CC$
can be pictured as

$$\hskip-11cm
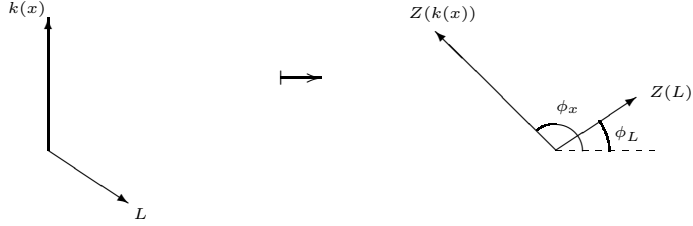
\begin{picture}(-100,100)
\put(0,20){\vector(0,1){50}}
\put(0,20){\vector(3,-2){30}}
\put(-15,72){\tiny\mbox{$k(x)$}}
\put(32,-6){\tiny\mbox{$L$}}
\put(85,45){\mbox{$\mapsto$}}
\put(190,20){\vector(-1,1){45}}
\put(190,20){\vector(3,2){30}}
\put(135,70){\tiny\mbox{$Z(k(x))$}}
\put(225,40){\tiny\mbox{$Z(L)$}}
\put(190,30){\qbezier(-7.2,-2.3)(-3.3,0.3)(0,0)}
\put(190,20){\oval(20,20)[rt]}
\put(190,35){\tiny\mbox{$\phi_x$}}
{\linethickness{0.005mm}\put(190,20){\line(1,0){2}}
\put(195,20){\line(1,0){2}}\put(200,20){\line(1,0){2}}\put(205,20){\line(1,0){2}}\put(210,20){\line(1,0){2}}
\put(215,20){\line(1,0){2}}\put(220,20){\line(1,0){2}}\put(225,20){\line(1,0){2}}}
\put(190,30){\qbezier(16.3,1.1)(19.9,-4.5)(20,-10)}
\put(212,25){\tiny\mbox{$\phi_L$}}
\end{picture}
$$
\vskip0.5cm
which shows that $Z$ can be seen as an orientation preserving automorphism of $\RR^2$ (under the identification of $N(C)$
with $\ZZ^2$ as chosen above).  But then, by composing  with a matrix in ${\rm GL}\!{}^+(2,\RR)$, it can be turned  into
$-\deg+i\cdot\rk$, which is rotation by $\frac{\pi}{2}$ anti-clockwise. In other words, in the $\widetilde{\rm GL}\!{}^+(2,\RR)$-orbit of any numerical stability condition $\sigma$ one finds one sta\-bili\-ty condition, say $\sigma'$,
with stability function $Z=-\deg+i\cdot\rk$. Moreover,
we may assume that for all points $x\in C$ the phase of $k(x)$ is $\phi_x=1$ and hence $L$ will have
phase $\phi_L\in(0,1)$. (A priori the phases of $k(x_1)$ and $k(x_2)$ for two distinct points could be different. But since
the two sheaves are numerically equivalent, one has $\phi_{x_1}-\phi_{x_2}\in2\ZZ$. Using $\phi_{x_i}-1<\phi_L<\phi_{x_i}$, $i=1,2$,
for an arbitrary line bundle $L$, one finds that in fact $\phi_{x_1}=\phi_{x_2}$.)
Thus, the heart $\ka'=\kp'(0,1]$ of $\sigma'$ contains all $k(x)$ and all line bundles $L\in\Pic(C)$. Since any coherent
sheaf on $C$ admits a filtration with quotients either isomorphic to point sheaves or line bundles, this shows
$\coh(C)\subset\ka'$ and hence $\coh(C)=\ka'$, as both are hearts of bounded t-structures on $\Db(C)$. 
\end{proof}

\begin{lem}\label{lem:keycurves}
Suppose $\sigma=(\kp,Z)$ is a numerical stability condition on $\Db(C)$. Then all point sheaves
$k(x)$, $x\in C$, and all line bundles $L\in\Pic(C)$ are stable with respect to $\sigma$.
\end{lem}

\begin{proof} The following is a simplified version of an argument in \cite{GR}.

i) Consider $E\in\coh(C)$ and an exact triangle $A\to E\to B$ in $\Db(C)$ with $\Hom^{<0}(A,B)=0$. {\it Claim}: Then $A\cong A_0\oplus A_1[-1]$ and $B\cong B_0\oplus B_{-1}[1]$ with $A_i,B_i\in\coh(C)$.

Write $A=\bigoplus A_i[-i]$ and $B=\bigoplus B_i[-i]$. The long cohomology sequence of the exact triangle yields
an exact sequence \begin{equation}\label{eqn:les}\xymatrix{0\ar[r]& B_{-1}\ar[r]^\varphi& A_0\ar[r]& E\ar[r]& B_0\ar[r]^\psi& A_1\ar[r]& 0}\end{equation} and $B_{i-1}\cong A_i$ for $i\ne0,1$.
If $A_i\ne0$ for some $i\ne0,1$, then $$0\ne\Hom(A_i[-i],B_{i-1}[-i])=\Hom^{-1}(A_i[-i],B_{i-1}[-(i-1)])$$ and hence
$\Hom^{-1}(A,B)\ne0$, which contradicts the assumption. 

ii) Consider $E\in\coh(C)$ and an exact triangle $A\to E\to B$ in $\Db(C)$ with $\Hom^{\leq0}(A,B)=0$.
{\it Claim}: Then $A,B\in\coh(C)$. (Here one uses the assumption $g(C)>0$.)

Use the long exact sequence (\ref{eqn:les}). If $\varphi\ne0$, then twisting  it
with sections of $\omega_C$ yields non-trivial  $B_{-1}\to A_0\otimes\omega_C$. Hence, by Serre
duality,  $\Hom^{1}(A_0,B_{-1})\ne0$ and, therefore, $\Hom(A,B)\ne0$, which contradicts the assumption. Thus,
$\varphi=0$ and hence $B_{-1}=0$. The argument for proving $\psi=0$ is similar.

iii) {\it Claim}: All point sheaves $k(x)$ and all line bundles $L$ are semistable.

Apply ii) to $E=k(x)$ and $E=L$ by letting $A$ be the first semistable factor. Then,
$\Hom^{\leq0}(A,B)=0$ and by ii) $A,B\in\coh(C)$. For $E=k(x)$ this immediately yields $B=0$, as
$k(x)$ has no proper non-trivial subsheaves. For $E=L$, the subsheaf $A$ must also be a line bundle
and hence $B$ a torsion sheaf (or trivial). But if $B\ne0$, then $\Hom(A,B)\ne0$. Contradiction.

iv) {\it Claim}: All point sheaves $k(x)$ and all line bundles $L$ are stable.

We apply again ii) to $E=k(x)$ and $E=L$. Let $A_0$ be a stable factor of $E$ with $\Hom(A_0,E)\ne0$.
Then there exists an exact triangle $A\to E\to B$ with $A,B$ semistable and such that
all stable factors of $A$ are isomorphic to $A_0$ and $\Hom(A,B)=0$, cf.\ Exercise
\ref{exer:allsta}. (The vanishing of 
$\Hom^{<0}(A,B)$ follows directly from semistability.) Thus, $A,B\in\coh(C)$ and as in iii) this
shows $B=0$, i.e.\ all stable factors of $E$ are isomorphic to $A_0$. Hence, $[E]=n[A_0]$,
where $n$ is the number of stable factors. Since $[k(x)]=(0,1)$ and $[L]=(1,\deg(L))$, one must have
$n=1$, i.e.\ $k(x)$ and $L$ are stable.
\end{proof}

\begin{remark}
i) The minimal objects in $\coh(C)$ are the point sheaves $k(x)$. In the tilts of $\coh(C)$ occuring as hearts
of the other stability conditions on $\coh(C)$ the description is more complicated, see e.g.\ \cite[Lem.\ 2.4]{Wolf1}.

ii) For $g=1$ the quotient of ${\rm Stab}(C)$ by the natural left action of the group $\Aut(\Db(C))$ of all exact
$k$-linear autoequivalences can be described as the quotient of  ${\rm GL}\!{}^+(2,\RR)$
by ${\rm SL}(2,\ZZ)$ which can also be interpreted as a $\CC^*$-bundle over the moduli space of elliptic curves
(see \cite{BrAnn}).
Indeed, $\Aut(\Db(C))$ compares to ${\rm GL}\!{}^+(2,\RR)$ via the diagram 
$$\xymatrix{0\ar[r]&\ZZ[2]\times(C\times\widehat C)\ar[r]&\Aut(\Db(C))\ar[d]\ar[r]&{\rm SL}(2,\ZZ)\ar@{^(->}[d]\ar[r]&0\\
0\ar[r]&\ZZ[2]\ar[r]&\widetilde{\rm GL}\!{}^+(2,\RR)\ar[r]&{\rm GL}\!{}^+(2,\RR)\ar[r]&0.}$$
Note that $C\times\widehat C$ acts, by translation resp.\ tensor product, trivially on ${\rm Stab}(C)$.

iii) Stability conditions on $\Db(\PP^1)$ can also be described.  It turns out that ${\rm Stab}(\PP^1)$ is
connected and  simply connected, see \cite{Ma}, and in fact ${\rm Stab}(\PP^1)\cong\CC^2$, see \cite{Ok}.
 
\end{remark}

\section{Stability conditions on surfaces}\label{sec:stabonsurf}

We consider a smooth projective surface $X$ over a field
$k=\bar k$, but we shall also be interested in
compact complex surfaces. For any ample class $\omega\in{\rm NS}(X)\otimes\RR$ (or
a K\"ahler class $\omega\in H^{1,1}(X,\RR)$) one defines the degree and the slope 
(if $\rk(E)\ne0$) as $$\deg_\omega(E):=({\rm c}_1(E).\omega)\text{~~resp.~~}
 \mu_\omega(E):=\frac{\deg_\omega(E)}{\rk(E)}.$$ 
Fix $\beta\in\RR$ and think of it as $\beta=(B.\omega)$ for some $B\in{\rm NS}(X)\otimes\RR$.
Then consider the full additive subcategories $\kt_{(\omega,\beta)},\kf_{(\omega,\beta)}\subset\coh(X)$:
$$\kt_{(\omega,\beta)}:=\{E\in\coh(X)~|~\forall\, E\twoheadrightarrow F\ne0~\text{torsion free}:~\mu_\omega(F)>\beta\}$$
and 
$$\kf_{(\omega,\beta)}:=\{E\in\coh(X)~|~E~\text{torsion  free and}~\forall\,0\ne F\subset E:~\mu_\omega(F)\leq\beta\}.$$
 
We leave it as an exercise to verify that this defines a torsion theory for $\coh(X)$. One needs to use the existence of the Harder--Narasimhan filtration and the fact that the saturation $F\subset F'\subset E$ of a subsheaf
$F\subset E$, i.e.\ the minimal $F'$ such that
$E/F'$ is torsion free, satisfies $\mu_\omega(F)\leq\mu_\omega(F')$.
Note that for the existence of the Harder--Narasimhan filtration it would
be enough to assume that $\omega$ is nef. (But Jordan--H\"older filtrations need a stronger positivity.)

The tilt of $\coh(X)$ with respect to the torsion theory
$(\kt_{(\omega,\beta)},\kf_{(\omega,\beta)})$ is denoted
$$\ka(\omega,\beta)\subset\Db(X):=\Db(\coh(X))$$
or $\ka(\exp(B+i\omega))$ if $\beta=(B.\omega)$.

\begin{ex}
For a closed point $x\in X$ the skyscraper sheaf
$k(x)$ is contained in $\ka(\omega,\beta)$. If $E$ is  $\mu_\omega$-stable sheaf with $\mu_\omega(E)=\beta$, then
$E[1]\in\ka(\omega,\beta)$.
\end{ex}

\subsection{Classification of hearts}\label{sec:Classi}
For the following see \cite[Prop.\ 10.3]{BrK3}. 

\begin{thm}\label{thm:detheart}
If $\sigma$ is a numerical stability condition on $\Db(X)$ such that
all point sheaves $k(x)$, $x\in X$, are $\sigma$-stable of phase $1$, then
the heart $\ka_\sigma$ of $\sigma$ is of the form $\ka_\sigma=\ka(\omega,\beta)$ for some
nef class $\omega\in{\rm NS}(X)$ and $\beta\in\RR$.\footnote{In fact, one proves $(\omega.C)>0$ 
for all curves $C$. This is not quite enough to conclude that $\omega$ is ample. But if $\sigma$
is `good', i.e.\ contained in a maximal component, then $\omega$ will in addition be in the interior of the nef
cone and hence ample.}\label{good} 
\end{thm}

We will give the main technical arguments that go into the proof. They illustrate standard techniques
in the study of $\coh(X)$ and $\Db(X)$.

\smallskip
\noindent
{\bf Claim 1:} i) If $E\in\kp(0,1]$, then $H^i(E)=0$ for $i\ne0,-1$ and $H^{-1}(E)$ is torsion free.

ii) If $E\in\kp(1)$ is stable, then $E\cong k(x)$ for some $x\in X$ or $E[-1]\in\coh(X)$, which in addition is locally free.

\begin{proof}
i) It is enough to show the assertion for $E$ stable and not isomorphic to any point sheaf
$k(x)$. Then stability of $E$ and $k(x)$
implies $$\Hom^{<0}(E,k(x))=\Hom^{\leq 0}(k(x),E)=0.$$ Thus, by Serre duality $\Ext^i(E,k(x))=0$ for $i\ne0,1$.

Write $E\cong E^\bullet$ with $E^k$ locally free. Then $E^\bullet$ is not exact in degree $k$ in a point $x$ if and
only if $H^k(E^\bullet\otimes k(x))\ne0$. But this cohomology also computes $\Ext^{k}(E,k(x))^*$
and thus only $k=0,1$ are possible. This proves the first part of i).

Now consider the spectral sequence
\begin{equation}\label{eqn:ss}
E_2^{p,q}=\Ext^p(H^{-q}(E),k(x))\Rightarrow \Ext^{p+q}(E,k(x))
\end{equation}
for which $E_2^{p,q}=0$ if $p\ne0,1,2$. Let $i$ and $j$ be maximal resp.\ minimal with $H^i(E)\ne0\ne H^j(E)$.
Then the marked boxes in the picture of (\ref{eqn:ss}) survive the passage to $E_\infty$:
\vskip0.5cm
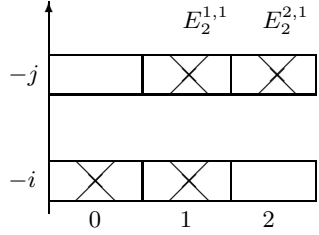
$$\hskip-3cm
\begin{picture}(-100,100)
\put(-105,70){\small\mbox{$-j$}}
\put(-105,30){\small\mbox{$-i$}}
\put(-90,80){\line(1,0){100}}
\put(-90,65){\line(1,0){100}}
\put(-90,25){\line(1,0){100}}
\put(-90,40){\line(1,0){100}}
\put(10,65){\line(0,1){15}}
\put(10,25){\line(0,1){15}}
\put(-22,65){\line(0,1){15}}
\put(-55,65){\line(0,1){15}}
\put(-22,25){\line(0,1){15}}
\put(-55,25){\line(0,1){15}}
\put(-45,25){\line(1,1){15}}
\put(-45,40){\line(1,-1){15}}
\put(-45,65){\line(1,1){15}}
\put(-45,80){\line(1,-1){15}}
\put(-12,65){\line(1,1){15}}
\put(-12,80){\line(1,-1){15}}
\put(-80,25){\line(1,1){15}}
\put(-80,40){\line(1,-1){15}}
\put(-75,15){\small\mbox{$0$}}
\put(-41,15){\small\mbox{$1$}}
\put(-10,15){\small\mbox{$2$}}
\put(-40,90){\small\mbox{$E_2^{1,1}$}}
\put(-10,90){\small\mbox{$E_2^{2,1}$}}
\put(-90,20){\vector(0,1){80}}
\end{picture}
$$
Note that for $x\in{\rm supp}(H^i(E))$, clearly $E_2^{0,-i}\ne0$ and hence $E^{-i}\ne0$ which proves $i\in\{0,-1\}$
confirming what we have seen already. But the spectral sequence also proves the second part of i).
Indeed, suppose $H^{-1}(E)$ is not torsion free. If it has zero-dimensional torsion, then there exists a point $x\in X$ with
$\Hom(k(x),H^{-1}(E))\ne0$ and thus by Serre duality $E_2^{2,1}\ne0$ contradicting $E^3=0$.
If the torsion of $H^{-1}(E)$ is purely one-dimensional, then there exists a point $x\in X$
with $\Ext^1(k(x),H^{-1}(E))\ne0$ (use that $\Ext^1(k(x),\ko_C)\ne0$ for a curve $C$ and a smooth point $x\in C$).
Thus, by Serre duality, $E_2^{1,1}\ne0$ contradicting $E^2=0$. This concludes the proof of i).

For $k(x)\not\cong E\in\kp(1)$ stable one has $\Hom(E,k(x))=0$ (cf.\ Remark
\ref{rem:useful}, ii)) and hence 
$i=j=-1$ in (\ref{eqn:ss}), i.e.\ $E\cong F[1]$ for some torsion free $F\in\coh(X)$. Moreover, $F$ is locally free if and
only if $\Ext^1(k(x),F)=0$ for all $x\in X$. But $\Ext^1(k(x),F) \cong\Hom(k(x),E)=0$, again because
$E$ and $k(x)$ are non-isomorphic
stable objects of the same phase.
\end{proof}

\smallskip
\noindent
{\bf Claim 2:} i) If $E\in\coh(X)$, then $E\in\kp(-1,1]$.

ii) If $E\in\coh(X)$ is torsion, then $E\in\kp(0,1]$. (This part will be proved only later, but
logically the assertion belongs  here.)

\begin{proof}
For any $A\in\kp(\phi)$ Claim 1 shows that then $H^i(A)=0$ if $i\geq0$, $\phi>1$ or
if $i\leq0$, $\phi\leq-1$. Let $A\to E$ be the maximal semistable factor of $E$.
Then, $\Hom(H^i(A)[-i],E)\ne0$ for some $i$. But on the one hand, if
$\phi(A)>1$, then $H^i(A)=0$ for $i\geq0$, and on the other hand  $\Ext^{<0}=0$ on the abelian category
$\coh(X)$. Thus, $E\in\kp(\leq1)$.\footnote{An alternative proof can be given by using the spectral sequence $E_2^{p,q}=\Ext^p(H^{-q}(A),E)\Rightarrow \Ext^{p+q}(A,E)$.} A similar argument for the minimal semistable factor $E\to B$ proves
$E\in\kp(>\!-1)$. 
\end{proof}

\smallskip
\noindent
{\bf Claim 3:} Let $\kt:=\coh(X)\cap\kp(0,1]$ and $\kf:=\coh(X)\cap\kp(-1,0]$. Then
$(\kt,\kf)$ defines a torsion theory for $\coh(X)$ with tilt
$\kp(0,1]=\ka_\sigma$.

\begin{proof}
Apply Proposition \ref{prop:tiltssand} to compare the standard t-structure $\kd^{\leq0}\subset\Db(X)$ with
$\kd^{\leq0}_\sigma\subset\Db(X)$ given by $\sigma$. The assumptions $\kd_\sigma^{\leq0}\subset\kd$
and $\kd^{\leq-1}\subset\kd_\sigma^{\leq0}$ are easily verified. Here,
$\kd_\sigma^{\leq0}=\kp(>\!0)$.
\end{proof}

Note that so far we have not described the torsion theory $(\kt,\kf)$ explicitly,
but ii) in Claim 2 can now be proven easily. Indeed, decompose a torsion sheaf $E\in\coh(X)$
with respect to the torsion theory $(\kt,\kf)$
as $A\to E\to B$ with $A\in\kp(0,1]$ and $B\in\kp(-1,0]$. Then $H^0(B)$ is torsion free by Claim 1, i).
But the long cohomology sequence yields a surjection $E=H^0(E)\twoheadrightarrow H^0(B)$
and an exact sequence $H^1(E)\to H^1(B) \to H^2(A)$. Thus, if $E$  is torsion, then $H^0(B)=0$,
and, as $H^1(E)=0=H^2(A)$, also $H^1(B)=0$. Therefore,  $B=0$ and hence $E\in \kp(0,1]$.

\smallskip
For the following we need that the stability condition is numerical.
\noindent
{\bf Claim 4:} The imaginary part of $Z$ is of the form
$\langle (0,\omega,\beta),v(E)\rangle$ for some $\beta\in\RR$ and $\omega\in{\rm NS}(X)\otimes\RR$ with
$(\omega.C)>0$ for all curves $C\subset X$.

\begin{proof} 
Here, $v(E)={\rm ch}(E)\sqrt{{\rm td}(X)}$ is the Mukai vector of $E$, which
for K3 surfaces equals $(\rk(E),{\rm c}_1(E),\chi(E)-\rk(E))$. The Mukai pairing
$\langle~~,~~\rangle$ is the usual intersection pairing with a sign
in the pairing of $H^0$ and $H^4$. Since it is non-degenerate and $Z$ is numerical,
there exists a vector $w=(w_0,w_1,w_2)\in{\rm NS}(X)\otimes\CC$ such that $Z(E)=\langle w,v(E)\rangle$ for all $E$.
As $Z(k(x))\in\RR_{<0}$ and $v(k(x))=(0,0,1)$, one has $w_0\in\RR_{>0}$ and thus
${\rm Im}(w)=(0,\omega,\beta)$. 

Let $C\subset X$ be a curve. Then by Claim 2, $\ko_C\in\kp(0,1]$.
Since $v(\ko_C)=(0,[C],~~)$, this yields $(\omega.C)={\rm Im} (Z(\ko_C))\geq0$. If $(\omega.C)=0$, then
$Z(\ko_C)\in\RR_{<0}$ and hence $\ko_C\in\kp(1)$. 
The stable factors of $\ko_C$, which are all in $\kp(1)$, are by Claim 1 of the form  $k(x)$ or $E[1]$
with $E\in\coh(X)$ locally free. Since their ranks add up to $\rk(\ko_C)=0$, only point sheaves
$k(x)$ can in fact occur. But this would yield the contradiction $(0,[C],\ast)=v(\ko_C)=\sum v(k(x_i))=(0,0,n)$. 
\end{proof}

\smallskip
\noindent
{\bf Claim 5:} Let $\omega$ and $\beta$ be as before.
Suppose $E\in\coh(X)$ is torsion free and $\mu_\omega$-stable. Then

i) $E\in\kt$ if and only if $\mu_\omega(E)>\beta$.

ii) $E\in\kf$ if and only if $\mu_\omega(E)\leq\beta$.

\begin{proof}
Since $(\kt,\kf)$ is a torsion theory, there exists a short exact sequence 
$0\to A\to E\to B\to 0$ in $\coh(X)$ with $A\in\kt$ and $B\in\kf$. By Claim 2 all torsion sheaves are in $\kt$.
Thus, since $\Hom(\kt,\kf)=0$, the sheaf $B$ must be torsion free.

By definition of $\kt$ one has $A\in\kp(0,1]$ and thus ${\rm Im}(Z(A))\geq0$ which is equivalent to $\mu_\omega(A)\geq\beta$. Similarly, for $B\in\kf$, one has $B\in\kp(-1,0]$ and hence ${\rm Im}(Z(B))\leq0$ or, equivalently, $\mu_\omega(B)\leq\beta$.
If both $A\ne0\ne B$, this would contradict   $\mu_\omega$-stability of $E$. Thus, $E=A\in\kt$ or $E=B\in\kf$.
Clearly, if $\mu_\omega(E)>\beta$, then $E\in \kt$. Similarly, if $\mu_\omega(E)<\beta$, then $E\in\kf$.
Thus, only the case $\mu_\omega(E)=\beta$, which is equivalent to ${\rm Im}(Z(E))=0$, remains to be settled.
Suppose $E\in\kt$, then $E\in\kp(1)$ and by Claim 1 all stable factors would be of the form $k(x)$ or $F[1]$ with $F$ locally free.
Thus, the sum of the ranks would be $\leq0$ contradicting the assumption that $E$ is torsion free.
\end{proof}

This finishes the proof of Theorem \ref{thm:detheart}. Indeed, from Claim 2, ii) and Claim 5
one deduces $\kt_{(\omega,\beta)}\subset\kt$ and $\kf_{(\omega,\beta)}\subset\kf$. Since
both, $(\kt,\kf)$ and $(\kt_{(\omega,\beta)},\kf_{(\omega,\beta)})$, are torsion theories for $\coh(X)$,
they actually coincide. But by Claim 3, the tilt of $\coh(X)$ with respect to $(\kt,\kf)$ is $\ka_\sigma$, which,
therefore, coincides with the the tilt of $\coh(X)$ with respect to $(\kt_{(\omega,\beta)},\kf_{(\omega,\beta)})$. But the
latter is by definition just $\ka(\omega,\beta)$.\qqed

\begin{cor}
If $X$ is a smooth projective surface, then $\coh(X)$ cannot be the heart of a numerical stability condition. 
\end{cor}

\begin{proof} Suppose $\coh(X)$ is the heart of a stability condition.
Since point sheaves $k(x)\in\coh(X)$ do not contain proper subsheaves, they are automatically stable.
Now use $Z(I_{x_1,\ldots,x_n})=Z(\ko_X)-\sum Z(k(x_i))=Z(\ko_X)-nZ(k(x))$ which would be
contained in the lower half plane $-\HH$ except if $Z(k(x))\in\RR_{<0}$. Thus, all point sheaves $k(x)$ are stable
of phase one. By Theorem \ref{thm:detheart}, this shows $\coh(X)=\ka(\omega,\beta)$.
But on a projective surface there always exists a $\mu$-stable vector bundle $E\in\coh(X)$
with $\mu_\omega(E)\leq\beta$ and thus $E[1]\in\ka(\omega,\beta)$.\footnote{The argument
is still valid although $\omega$ only satisfies $(\omega.C)>0$.} Contradiction.
\end{proof}

\begin{remark}
Compare the corollary with Claim 3 which in particular shows that $\coh(X)$ is the tilt of the heart
of a stability condition. Indeed, $\coh(X)[1]$ is the tilt of $\ka(\omega,\beta)$
with respect to the torsion theory $\kt_\ka:=\kf[1]$ and $\kf_\ka:=\kt$.
So, not any tilt of the heart of a stability condition is again the heart of a stability condition.
The reason in our specific example is that $Z(\kf_\ka)$ `spreads out' over $\HH$, e.g.\ $Z(k(x))\in\RR_{<0}$ 
whereas ${\rm Re}(Z(E_n))/{\rm Im}(Z(E_n))\to\infty$ for $E_n$ a sequence of $\mu$-stable bundles
with fixed slope $\mu>\beta$ but ${\rm c}_2(E_n)\to\infty$.
\end{remark}

\begin{remark}
If $X$ is a K3 surface with $\Pic(X)=0$, in particular $X$ is not projective, then $\coh(X)$ is the heart of
a stability condition. As a stability function in this case one can choose $Z(E)=-r(\alpha+i\beta)-s$
with $\beta<0$, where $v(E)=(r,0,s)$. See \cite{HMSComp}.
\end{remark}

\begin{remark} For the following observations see \cite{HuyJAG}.

i) Recall that the only minimal objects in $\coh(X)$ are the point sheaves $k(x)$. For K3 surfaces
 the minimal objects in the tilted category $\ka(\omega,\beta)$ are of the form $k(x)$ and $E[1]$, where $E$ is a $\mu$-stable vector
bundle with $\mu_\omega(E)=\beta$. 

ii) By a classical theorem of Gabriel, $\coh(X)\cong\coh(X')$ (as $k$-linear categories) if and only
if $X\cong X'$ (as varieties over $k$). But there exist non isomorphic (K3) surfaces
with equivalent derived categories $\Db(X)\cong\Db(X')$ (as $k$-linear triangulated categories).
For projective K3 surfaces derived equivalence is determined by the tilted abelian categories:
$\Db(X)\cong\Db(X')$ if and only if there exist $(\omega,\beta)$ and $(\omega',\beta')$ on $X$ resp.\
$X'$ with $\ka(\omega,\beta)\cong\ka(\omega',\beta')$.
\end{remark}
\subsection{Construction of hearts}\label{sec:constrhearts}
For fixed $\omega$ and $\beta=(B.\omega)$ as before, we consider the function
$$Z(E):=\langle \exp(B+i\omega),v(E)\rangle.$$
Here, $v(E)={\rm ch}(E)\sqrt{{\rm td}(X)}=(r,\ell,s)$ and $\exp(B+i\omega)=(1,B+i\omega,\frac{B^2-\omega^2}{2}+i(B.\omega))$.

For the following  assume that $X$ is a K3 surface.\footnote{For arbitrary $X$ one would
need to assume $\omega^2\gg0$.}

\begin{prop}\label{prop:ZEsph}
If $\omega\in{\rm NS}(X)\otimes\RR$ is ample and such that $Z(E)\not\in\RR_{\leq0}$
for all spherical $E\in\coh(X)$, then $Z$ is a stability function on the abelian category $\ka(\omega,\beta)$.
\end{prop}

\begin{proof} As we will see, the assumption $Z(E)\not\in\RR_{\leq0}$ is satisfied whenever $\omega^2>2$.
For the definition of `spherical' see Section \ref{sec:autoexa}.

Since $Z$ is linear and $\ka(\omega,\beta)$ is the tilt of $\coh(X)$ with respect to
$(\kt_{(\omega,\beta)},\kf_{(\omega,\beta)})$, it suffices to show that $Z(E)\in\overline\HH$ for
$E\in\kt_{(\omega,\beta)}$ and $-Z(E)\in\overline\HH$ for $E\in\kf_{(\omega,\beta)}$.

By definition of the torsion theory,  ${\rm Im}(Z(E))=\langle (0,\omega,\beta),v(E)\rangle=(\omega.\ell)-r\beta$,
which is $\geq0$ for $E\in\kt_{(\omega,\beta)}$
and $\leq0$ for $E\in\kf_{(\omega,\beta)}$. Moreover,
if $E\in\kt_{(\omega,\beta)}$ is not torsion, then the inequality is strict.
If $0\ne E\in\coh(X)$ is supported in dimension zero, then
${\rm Im}(Z(E))=0$, but ${\rm Re}(Z(E))=-s=-h^0(E)<0$. Hence, $Z(E)\in\RR_{<0}\subset\overline\HH$.
If $0\ne E\in\kt_{(\omega,\beta)}$ has support in dimension one, then  ${\rm Im}(Z(E))>0$. So
lets consider $E\in\kf_{(\omega,\beta)}$. One easily reduces to the case
that $E$ is $\mu$-stable with slope $\mu_\omega(E)=\beta$. In particular, $(\ell-rB)\in\omega^\perp$
and thus, by Hodge index theorem, $(\ell-rB)^2\leq0$. Now one rewrites
the real part of $Z(E)$ as
$${\rm Re}(Z(E))=(B.\ell)-s-\frac{B^2-\omega^2}{2}r=\frac{1}{2r}\left((\ell^2-2rs)+r^2\omega^2-(\ell-rB)^2\right).$$
Clearly, $r^2\omega^2-(\ell-rB)^2>0$. And for the remaining summand observe
that $(\ell^2-2rs)=v(E)^2=-\chi(E,E)=-2+\ext^1(E,E)\geq-2$ for any stable $E$. Thus,
${\rm Re}(Z(E))\in \RR_{>0}$ whenever $E$ is not rigid (or $r\omega^2>2$). More precisely, it suffices
to assume $Z(E)\not\in\RR_{\leq0}$ for all spherical $E\in\coh(X)$ to ensure that $Z(E[1])\in\overline\HH$ for
all $E\in\kf_{(\omega,\beta)}$.
\end{proof}

We let $\omega$ be ample  and assume that $Z$ defines a stability function on
$\ka(\omega,\beta)$.
\begin{prop}\label{prop:discZ}
For rational $\omega,B\in{\rm NS}(X)\otimes\QQ$, the stability function $Z=\langle\exp(B+i\omega),~~\rangle$ on $\ka(\omega,\beta)$
satisfies the Harder--Narasimhan pro\-per\-ty and the local finiteness condition (cf.\ Definition \ref{dfn:locfinite}). In particular, finite Jordan--H\"older filtrations exist.
\end{prop}

\begin{proof}
In order to show the existence of Harder--Narasimhan filtrations for all objects in $\ka:=\ka(\omega,\beta)$, 
it suffices to show that any descending filtration $\ldots\subset E_{i+1}\subset E_i\subset\ldots \subset E_0=E$
with increasing phases $\phi(E_{i+1})>\phi(E_i)>\ldots>\phi(E_0)$ stabilizes   and similarly for chains of quotients
(cf.\ \cite[Prop.\ 2.4]{BrAnn}).

Since $B$ and $\omega$ are rational, the image $Z(\ka)\subset \CC$ is discrete. Indeed, if $B,\omega\in\frac{1}{m}{\rm NS}(X)$, then $m^2Z(\ka)\subset\ZZ[i]$. Using linearity of $Z$, one deduces that
${\rm Im}(Z(E_{i+1}))\leq{\rm Im}(Z(E_i))$ for all $i$ and hence, by discreteness of $Z(\ka)$, ${\rm Im}(Z(E_i))\equiv {\rm const}$ for
$i\gg0$. But then ${\rm Im}(Z(E_i/E_{i+1}))=0$ which is a contradiction.
\vskip0.5cm
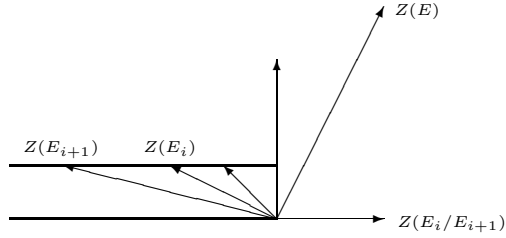
$$\hskip-3cm
\begin{picture}(-100,100)
\put(10,25){\vector(0,1){60}}
\put(10,25){\vector(1,2){40}}
\put(10,25){\vector(-1,1){20}}
\put(10,25){\vector(-2,1){40}}
\put(10,25){\vector(-4,1){80}}
{  \linethickness{0.015mm}
\put(10,25){\vector(1,0){40}}}
\put(55,22){\tiny\mbox{$Z(E_i/E_{i+1})$}}
\put(54,102){\tiny\mbox{$Z(E)$}}
\put(-85,50){\tiny\mbox{$Z(E_{i+1})$}}
\put(-40,50){\tiny\mbox{$Z(E_{i})$}}
\put(10,45){\line(-1,0){100}}
\put(10,25){\line(-1,0){100}}
\end{picture}
$$
The argument to deal with chains of quotients is trickier, see \cite[Prop.\ 7.1]{BrK3}.
The local finiteness is again elementary (see \cite[Lem.\ 4.4]{BrK3}): Fix any $\varepsilon<\frac{1}{2}$.
Then for any $\phi$ the image of $Z:\kp(\phi-\varepsilon,\phi+\varepsilon)\to\CC$ will be contained in a region
strictly smaller than a half plane. Thus, subobjects and quotients of a fixed $E\in\kp(\phi-\varepsilon,\phi+\varepsilon)$
will have image in a bounded region, which for $Z$ discrete implies the finiteness of ascending and descending chains.

\vskip0.5cm
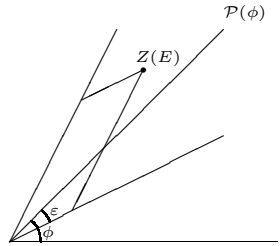
$$\hskip-3cm
\begin{picture}(100,100)
\put(10,25){\line(1,1){80}}
\put(10,25){\line(2,1){80}}
\put(60,90){\line(-2,-1){23}}
\put(10,25){\line(1,2){40}}
\put(60,90){\line(-1,-2){26.5}}
\put(25,35){\tiny\mbox{$\varepsilon$}}
\put(90,110){\tiny\mbox{$\kp(\phi)$}}

\put(22,27){\tiny\mbox{$\phi$}}
\put(10,25){\qbezier(8,8)(11.2,6)(11.5,0)}
\put(10,25){\qbezier(12,12)(14,10)(15,7.5)}
{  \linethickness{0.0015mm}
\put(10,25){\line(1,0){100}}}
\put(60,90){\circle*{2}}
\put(57,93){\tiny\mbox{$Z(E)$}}
\end{picture}
$$
\vskip-1cm
\end{proof}

\begin{cor}\label{cor:constrstable}
If $\omega,B\in{\rm NS}(X)\otimes\QQ$ with $\omega$ ample and such that $Z(E)\not\in\RR_{\leq0}$ for
all spherical $E\in\coh(X)$ (e.g.\ $\omega^2>2$), then $Z$ defines a stability condition
on $\ka(\omega,\beta)$ and, therefore, a stability condition on $\Db(X)$ with heart $\ka(\omega,\beta)$.\qqed
\end{cor}

In fact, the same result holds for real $\omega,B\in{\rm NS}(X)\otimes\RR$, but this is proved a
posteriori, after a detailed analysis of the component of ${\rm Stab}(X)$ containing the stability conditions with rational $Z$ constructed above.

\section{The topological space of stability conditions}

This lecture is devoted to fundamental results of Bridgeland concerning the 
topological structure of the space of stability conditions. Since a sta\-bility condition
consists of a slicing and a stability function, the natural reflex is to construct
topologies on the the space of slicings and on the space of stability functions separately
(the latter being a linear space) and use the product topology. This is roughly what happens, but details
are intricate. We will only discuss the main aspects, in Section \ref{sec:topslic} for the space of
slicings and in Section \ref{sec:topstab} for ${\rm Stab}(\kd)$. The main result about the projection
from stability conditions to stability functions being a local homeomorphisms can be found in Section
\ref{sec:main}.

\subsection{Topology of ${\rm Slice}(\kd)$}\label{sec:topslic} As before, $\kd$ will denote a triangulated
ca\-te\-gory.
The easier case of an abelian category is left  to the reader.
Recall that a slicing $\kp$ of $\kd$ consists of full additive subcategories
$\kp(\phi)\subset\kd$, $\phi\in\RR$ satisfying the conditions i)-iii) in Definition \ref{def:slic}.
In particular, any object $E\in\kd$ has a filtration by exact triangles with
factors $A_1\in\kp(\phi_1),\ldots,A_n\in\kp(\phi_n)$ such that $\phi_1>\ldots>\phi_n$. We
will also write $\phi^+(E):=\phi_1$ and $\phi^-(E):=\phi_n$ (or, if the dependence
on the given slicing needs to be stressed, $\phi^\pm_\kp(E)$). With this notation, $\kp(I):=\{E\in\kd~|~\phi^{\pm}(E)\in I\}$ for any interval $I\subset\RR$. The set of all slicings of $\kd$ is denoted by ${\rm Slice}(\kd)$.

\begin{definition}
For $\kp,\kq\in{\rm Slice}(\kd)$ one defines
$$d(\kp,\kq):={\rm sup}\{|\phi_\kp^\pm(E)-\phi_\kq^\pm(E)|~|~0\ne E\in\kd\}.$$
\end{definition}
Obviously, $d(\kp,\kq)\in[0,\infty]$.

\smallskip
\noindent
{\bf Claim 1:} With this definition, $d(~~,~~)$ is a generalized metric, i.e.\ it satisfies all the axioms for a distance function
except that the distance can be $\infty$.

\begin{proof}
The symmetry and the triangle inequality are straightforward to check. Suppose $d(\kp,\kq)=0$.
Then $E\in\kp(\phi)$, which is equivalent to $\phi^\pm_\kp(E)=\phi$, implies $\phi_\kq^\pm(E)=\phi$ and
thus $E\in\kq(\phi)$. Therefore, $\kp(\phi)\subset\kq(\phi)$. Reversing the role of $\kp$ and $\kq$ yields $\kp=\kq$.
\end{proof}

Thus, $d(~~,~~)$ defines a topology on ${\rm Slice}(\kd)$. Note that two slicings $\kp,\kq$ with
$d(\kp,\kq)=\infty$ are contained in different connected components. Indeed, 
$U_\kp:=\{\kr~|~d(\kp,\kr)<\infty\}$ is by definition an open neighourhood of $\kp$.
An open neighbourhood $U_\kq$ of $\kq$ is defined similarly.  Then use the triangle inequality to
show that $U_\kp\cap U_\kq=\emptyset$ and that ${\rm Slice}(\kd)\setminus U_\kp$ is open.
Hence, $U_\kp$ and $U_\kq$ are connected components of ${\rm Slice}(\kd)$.

\smallskip
\noindent
{\bf Claim 2:} $d(\kp,\kq)=\inf\{\varepsilon~|~\kq(\phi)\subset\kp[\phi-\varepsilon,\phi+\varepsilon]~\forall \phi\in\RR\}$.

\smallskip

This is best done as an exercise, but here is the argument anyway: Lets write $d:=d(\kp,\kq)$ and  $\delta:=\inf\{~~\}$.
The inequality $d\geq\delta$ is straightforward. For the other direction, one
has to show that $|\phi^\pm_\kp(E)-\phi^\pm_\kq(E)|\leq\delta$ for all $E\in\kd$. Consider the Harder--Narasimhan
filtration of $E$ with respect to $\kq$ and denote its semistable factors by $A_1,\ldots, A_n$. Then
$\phi^+_\kq(E)=\phi_\kq(A_1)$ and $\phi^-_\kq(E)=\phi_\kq(A_n)$. By definition of $\delta$, one has 
$\phi^+_\kp(A_i)\leq\phi_\kq(A_i)+\delta$ and $\phi^-_\kp(A_i)\geq\phi_\kq(A_i)-\delta$. But $\phi^+_\kp(E)\leq\max\{\phi^+_\kp(A_i)\}$ and $\phi^-_\kp(E)\geq\min\{\phi_\kp^-(A_i)\}$. Hence, $\phi^+_\kp(E)\leq\phi^+_\kq(E)+\delta$
and $\phi^-_\kp(E)\geq\phi^-_\kq(E)-\delta$. Continuing along these lines eventually yields
$|\phi^\pm_\kp(E)-\phi^\pm_\kq(E)|\leq\delta$ and, therefore, $d\leq \delta$.

\begin{prop} If for  stability conditions
$\sigma=(\kp,Z)$ and $\tau=(\kq,W)$ one has $W=Z$ and $d(\kp,\kq)<1$, then $\sigma=\tau$.
See \cite[Lem.\ 6.4]{BrAnn}.
\end{prop} 
Note that the strict inequality is needed, as e.g.\ $d(\kp,\kp[1])=1$.
$$\hskip-3cm
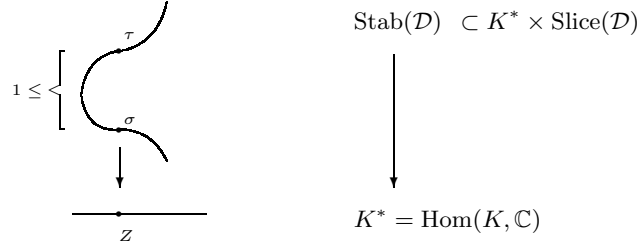
\begin{picture}(-100,100)
\put(20,80){\small\mbox{${\rm Stab}(\kd)$}}
\put(60,80){\small\mbox{$\subset K^*\times{\rm Slice}(\kd)$}}
\put(35,70){\vector(0,-1){50}}
\put(-67.6,35){\vector(0,-1){15}}
\put(-85,10){\line(1,0){50}}
\put(-68,71.5){\circle*{2}}
\put(-68,41.7){\circle*{2}}
\put(-90,42){\line(0,1){13}}
\put(-94.8,55.2){\tiny\mbox{$<$}}
\put(-90.2,42){\line(1,0){2}}
\put(-90.2,71.5){\line(1,0){2}}
\put(-90,71.5){\line(0,-1){13}}
\put(-66,44){\tiny\mbox{$\sigma$}}
\put(-66,76){\tiny\mbox{$\tau$}}
\put(-68,10){\circle*{2}}
\put(-68,0){\tiny\mbox{$Z$}}
\put(-108,55){\tiny\mbox{$1\leq$}}
\put(-50,80){\qbezier(-16.3,-8.3)(-3.3,-5.3)(0,10)}
\put(-50,80){\qbezier(-32,-25)(-30,-9.6)(-16.3,-8.3)}
\put(-50,80){\qbezier(-16.3,-38)(-30,-40)(-32,-25)}
\put(-50,80){\qbezier(0,-50)(-5,-39)(-16.3,-38.1)}
\put(20,5){\small\mbox{$K^*={\rm Hom}(K,\CC)$}}
\end{picture}
$$
\vskip0.5cm

\begin{proof}
Suppose there exists a $E\in\kp(\phi)\setminus\kq(\phi)$.

i) If $E\in\kq(>\!\phi)$, then by assumption $E\in\kq(\phi,\phi+1)$, i.e.\ $W(E)$ is contained in the half plane
left to $Z(E)$ which contradicts $W=Z$.

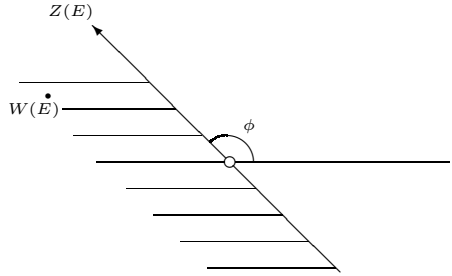
$$\hskip-9cm
\begin{picture}(50,100)

{  \linethickness{0.055mm}
\put(185,20){\line(1,0){80}}
\put(183,20){\circle{4}}}
{ \linethickness{0.005mm}
\put(184.33,18.5){\line(1,-1){40}}
}
\put(188,32){\tiny\mbox{$\phi$}}
\put(182,30){
\qbezier(-6.3,-2.3)(-3.3,0.3)(0,0)}
\put(182,20){\oval(20,20)[rt]}
{  \linethickness{0.015mm}
\put(181.5,21.5){\vector(-1,1){50}}
\put(115,75){\tiny\mbox{$Z(E)$}}
\put(100,38){\tiny\mbox{$W(E)$}}
\put(115,45){\circle*{2}}
\put(222.5,-20){\line(-1,0){48}}
\put(212.5,-10){\line(-1,0){48}}
\put(202.5,0){\line(-1,0){48}}
\put(192.5,10){\line(-1,0){48}}
\put(181,20){\line(-1,0){48}}
\put(172.5,30){\line(-1,0){48}}
\put(162.5,40){\line(-1,0){42}}
\put(152.5,50){\line(-1,0){48}}
}
\end{picture}
$$
\smallskip
\vskip1cm
ii) For $E\in\kq(<\phi)$ the argument is similar.

iii) The remaining case is $E\in\kq(\phi-1,\phi+1)$. Then there exists an exact triangle
$$A\to E\to B$$ with $A\in\kq(\phi,\phi+1)$ and $B\in\kq(\phi-1,\phi]$. If $A\in\kp(\leq\!\phi)$, then
$A\in\kp(\phi-1,\phi]$, which leads to the same contradiction as above. Thus, there exists
$0\ne C\in\kp(\psi)$ with $\psi>\phi$ and a non-trivial morphism $C\to A$ which factors via $B[-1]$ as
$E\in\kp(\phi)$: 
$$\xymatrix{&C\ar[d]^{\tiny\ne0}\ar@{..>}[dr]^{=0}\ar@{..>}[dl]_g&&\\
B[-1]\ar[r]&A\ar[r]\ar[r]&E\ar[r]&B}$$
But now $B[-1]\in\kq(\phi-2,\phi-1]\subset\kp(\phi-3,\phi]$ and, therefore, $g=0$. Contradiction.
\end{proof}

\subsection{Topology of ${\rm Stab}(\kd)$}\label{sec:topstab}
Let ${\rm Stab}(\kd)$ be the space of stability conditions $\sigma=(\kp,Z)$ for which
$Z$ factors via  a fixed quotient $K(\kd)\twoheadrightarrow \Gamma$. For simplicity, we will usually 
assume $\Gamma\cong\ZZ^n$ and write $\Gamma^*=\Hom(\Gamma,\CC)$.
Thus, $${\rm Stab} (\kd)\subset\Gamma^*\times{\rm Slice}(\kd).$$
It is then tempting, in particular when $\Gamma\cong\ZZ^n$, to endow ${\rm Stab}(\kd)$ with
the product topology. This is essentially what will happen, but not quite. Bridgeland   defines
a topology on ${\rm Stab}(\kd)$ such that for any connected component
${\rm Stab}(\kd)^{\rm o}$ there exists a linear topological subspace $V^{\rm o}\subset\Gamma^*$
such that ${\rm Stab}(\kd)^{\rm o}\subset {\rm Slice}(\kd)\times V^{\rm o}$ with the induced topology.
In particular, if $\Gamma$ is of finite rank, then indeed the topology on each
connected component of ${\rm Stab}(\kd)$ is induced by the product topology. 

For a fixed stability condition $\sigma=(Z,\kp)\in{\rm Stab}(\kd)$ 
a genera\-lized norm $\|~~\|_\sigma:\Gamma^*\to[0,\infty]$ is defined by 
by $$\|U\|_\sigma:=\sup\left\{\frac{|U(E)|}{|Z(E)|}~|~E~\sigma\text{-semistable}\right\}.$$
Then let $V_\sigma:=\{U\in\Gamma^*~|~\|U\|_\sigma<\infty\}$.
It is easy to check that $\|~~\|_\sigma$ is indeed a generalized norm which becomes a finite norm
on $V_\sigma$. For example, if $\|U\|_\sigma=0$, then $U(E)=0$ for all $\sigma$-semistable $E$ and,
using the existence of the Harder--Narasimhan filtration for all objects in $\kd$, this
implies $U=0$.

\begin{remark}
The approach of Kontsevich and Soibelman (cf.\ \cite{KoSo}) to the topology on the space of stability conditions is slightly
different. Roughly, the notion in \cite{KoSo} is stronger and leads to Bridgeland stability conditions
with maximal $V_\sigma=\Gamma^*$.  However, in the case of the stability conditions on K3 surfaces
that have been introduced earlier and that will be studied further in the next section, both notions coincide a posteriori.

Here are a few details. In \cite{KoSo},
a stability condition $\sigma=(\kp,Z)$ is required to admit 
a quadratic form $Q$ on $\Gamma\otimes\RR$ such that: 

i) $Q$ is negative definite on ${\rm Ker}(Z)$ and 

ii) If $|\phi^+(E)-\phi^-(E)|<1$, then $Q([E])\geq0$. (In fact, this is only a weak version of what is really
 required in \cite{KoSo}.)

In particular: iii) all semistable $0\ne E$ satisfy $Q([E])\geq0$.

Note that i) and iii) together are equivalent to the existence of a constant $C$ such that for all semistable
$0\ne E$ one has $$\|[E]\|< C\cdot |Z(E)|,$$ which is called the \emph{support property}. (Here, $\|~~\|$ is an arbitrary norm on $\Gamma\otimes\CC$ which is assumed to be finite dimensional.)

The support property automatically implies $V_\sigma=\Gamma^*$. Indeed, 
since $U\in\Gamma^*$ is in particular continuous, there exists a constant $D$ such that
$|U(\alpha)|\leq D\cdot\|\alpha\|$ for all $\alpha\in\Gamma$. Combined with the support property, it shows
that for any semistable $0\ne E$ one has $|U(E)|<C\cdot D\cdot|Z(E)|$ and hence $\|U\|_\sigma<C\cdot D$.

The actual condition ii) in \cite{KoSo} implies the finiteness condition that we have avoided so far, see Definition
\ref{dfn:locfinite}. Indeed, the support property implies that $Z(\Gamma)\subset \CC$ is discrete, see
\cite[Sect.\ 2.1]{KS}, which in turn implies local finiteness  (cf.\ Remark \ref{rem:discZ}).

\end{remark}

Now, in order to define a topology on ${\rm Stab}(\kd)$ Bridgeland considers
for $\sigma=(\kp,Z)$ the set
$$B_\varepsilon(\sigma):=\{\tau=(\kq,W)~|~\|W-Z\|_\sigma<\sin(\pi\varepsilon),~d(\kp,\kq)<\varepsilon\}.$$

\begin{remark} The two inequalities are compatible in the following sense. If
$E$ is $\sigma$-semistable, then the condition $\|W-Z\|_\sigma<\sin(\pi\varepsilon)$
implies $|\phi_\tau(E)-\phi_\sigma(E)|<\varepsilon$, which confirms $d(\kp,\kq)<\varepsilon$.
Indeed, if $\varphi$ is as below, then $\sin(\pi\varphi)\leq\frac{|(W-Z)(E)|}{|Z(E)|}<\sin(\pi\varepsilon)$.
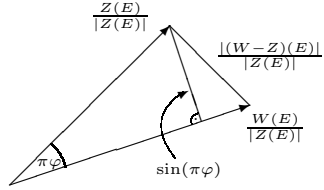
$$\hskip-3cm
\begin{picture}(-10,100)
\put(2,30){\qbezier(-6.3,5.3)(-2,1.3)(-1.2,-3.1)}
{\linethickness{0.05mm}\put(10,10){\qbezier(32,20)(20,35)(35,45.1)}
\put(46,55){\vector(2,1){1}}
\put(10,10){\qbezier(38.2,33)(37,35.7)(40.6,37)}}
\put(48.6,43.2){\tiny\mbox{$\cdot$}}
\put(-20,20){\vector(1,1){60}}
\put(-20,20){\vector(3,1){90}}
\put(39.8,80){\line(1,-3){12}}
\put(39.8,80){\line(1,-1){30}}
\put(-10,27){\tiny\mbox{$\pi\varphi$}}
\put(35,24){\tiny\mbox{$\sin(\pi\varphi)$}}
\put(10,82){\tiny\mbox{$\frac{Z(E)}{|Z(E)|}$}}
\put(57,67){\tiny\mbox{$\frac{|(W-Z)(E)|}{|Z(E)|}$}}
\put(68,40){\tiny\mbox{$\frac{W(E)}{|Z(E)|}$}}
\end{picture}
$$
\end{remark}

\smallskip
\noindent
{\bf Claim:} The sets $B_\varepsilon(\sigma)$ form the basis of a topology which on each connected component
${\rm Stab}(\kd)^{\rm o}$
coincides with the topology induced by the product topology under the inclusion ${\rm Stab}(\kd)\subset\Gamma^*\times{\rm Slice}(\kd)$.

\smallskip

The main technical point in the proof, is that for $\tau\in B_\varepsilon(\sigma)$
one always finds $C_1,C_2>0$ such that $C_1\cdot\|~~\|_\sigma\leq\|~~\|_\tau\leq C_2\cdot\|~~\|_\sigma$, i.e.\
$\|~~\|_\sigma\sim\|~~\|_\tau$. See \cite[Lem.\ 6.2]{BrAnn}.

Note that $V_\sigma=V_\tau$ for $\sigma,\tau$ contained in the same connected component
${\rm Stab}(\kd)^{\rm o}$.

\subsection{Main result}\label{sec:main}

We now come to the main result of \cite{BrAnn}.
Consider a connected component ${\rm Stab}(\kd)^{\rm o}\subset{\rm Stab}(\kd)$ of the space of locally finite (see
Definition \ref{dfn:locfinite}) stability conditions on $\kd$ with stability functions that factor via a fixed quotient $K(\kd)\twoheadrightarrow \Gamma$.
Let $V^{\rm o}\subset\Gamma^*$ be the associated linear space (i.e.\ $V^{\rm o}=V_\sigma$ for any $\sigma\in{\rm Stab}(\kd)^{\rm o}$).

\begin{thm}\label{thm:mainlocalhomeo}{\bf (Bridgeland)} 
The natural projection ${\rm Stab}(\kd)^{\rm o}\to V^{\rm o}$ is a local homeomorphism.
\end{thm} 

In fact, Bridgeland proves that for small $\varepsilon$ the projection yields a homeomorphism
 $B_\varepsilon(\sigma)\congpf B_{\sin(\pi\varepsilon)}(Z)$.

\begin{definition}\label{dfn:locfinite}
A stability condition $\sigma$ is \emph{locally finite} if there exists a $0<\eta$ such that
$\kp(\phi-\eta,\phi+\eta)$ is a category of finite type for all $\phi\in\RR$.
\end{definition} 
A category is of finite type if its artinian and noetherian, i.e.\ descending and ascending chains of
subobjects stabilize. In particular, all abelian $\kp(\phi)$ are of finite type which is equivalent to 
saying that any semistable object has a finite filtration with stable quotients.
 
The condition `locally finite' is needed when it comes to proving the main lemma
that says that for small $\varepsilon$ the inequality $\|W-Z\|_\sigma<\sin(\pi\varepsilon)$ implies
that there exists a stability condition $\tau=(\kq,W) \in B_\varepsilon(\sigma)$, i.e.\ such that
$d(\kp,\kq)<\varepsilon$. For the details of the proof we have to refer to the original \cite{BrAnn}.
  
 \begin{remark}\label{rem:discZ}
 The local finiteness of a stability condition holds automatically when the image of $Z$ is discrete
 \cite[Lem.\ 4.4]{BrK3}, cf.\ proof of Proposition \ref{prop:discZ}.
  
 \emph{Warning:} If a stability function $Z$ on the heart of a bounded t-structure is given, then knowing that
 the image of $Z$ discrete does not automatically give that $Z$ has the Harder--Narasimhan property.
 \end{remark}
 
 Examples of stability conditions that are not locally finite are easily constructed, see e.g.\ Example \ref{ex:patho}. 
 
 \begin{remark}
 A connected component ${\rm Stab}(\kd)^{\rm o}$ is called \emph{full} if $V^{\rm o}=\Gamma^*$.
 Then in \cite{Wolf2} it is shown that
 any full component is complete with respect to the metric
 $$d(\sigma,\tau):={\rm sup}_{0\ne E\in\kd}\left\{|\phi_\kp^\pm(E)-\phi_\kq^\pm(E)|,~|\log\frac{m_\sigma(E)}{m_\tau(E)}|\right\}$$
(which induces the topology on ${\rm Stab}(\kd)$ as described before).
Here, $\sigma=(\kp,Z)$, $\tau=(\kq,W)$ and $m_\sigma(E)$ is the mass of $E$, i.e.\ $m_\sigma(E)=\sum |Z(A_i)|$ if $A_i$ are the semistable factors of $E$.
\end{remark} 
\section{Stability conditions on K3 surfaces}\label{sec:K3}

The last lecture discusses results of Bridgeland on stability conditions for K3 surfaces. In Section
\ref{sec:stabonsurf}
explicit examples have been constructed and a characterization of those in terms of stability of point sheaves
has been proved. Applying autoequivalences of $\Db(X)$ produces more examples and Bridgeland in \cite{BrK3}
describes in this way a whole connected component of ${\rm Stab}(X)$. We will sketch the main steps of the argument
and present an intriguing conjecture, due to Bridgeland, predicting that this distinguished  connected component is
simply connected. At the end, we will rephrase the conjecture in terms of the fundamental group
of a certain moduli stack.

\subsection{Main theorem and conjecture}
For a triangulated category $\kd$, we have introduced ${\rm Stab}(\kd)$, the space
of locally finite stability conditions $\sigma=(\kp,Z)$ on $\kd$. Implicitly, the stability function
$Z:K(\kd)\to\CC$ is assumed to factor through a fixed quotient $K(\kd)\twoheadrightarrow \Gamma$, which
in the application is obtained by quotienting $K(\kd)$ by numerical equivalence.

Following \cite{BrAnn}, we have equipped ${\rm Stab}(\kd)$ with a topology which on each connected
component ${\rm Stab}(\kd)^{\rm o}\subset {\rm Stab}(\kd)$ is induced by the pro\-duct topology on
$\Gamma^*\times{\rm Slice}(\kd)$. Moreover, according to Theorem \ref{thm:mainlocalhomeo}, for each
${\rm Stab}(\kd)^{\rm o}$ there exists a linear subspace $V^{\rm o}\subset\Gamma^*$ such that
the projection ${\rm Stab}(\kd)^{\rm o}\to V^{\rm o}$ is a local homeomorphism. Ideally, this would be
a topological covering of an open subset of $V^{\rm o}$, but the covering property is known to be violated
in examples, see e.g.\ \cite{Mein}.
However, in the case we are interested in here, namely $\kd=\Db(X)$ with $X$ a smooth projective K3
surface, the covering property has been verified in \cite{BrK3}. But we first need to describe the image.

In Section \ref{sec:constrhearts}, stability conditions $\sigma=(\kp,Z)$ on $\Db(X)$ have been constructed
with $Z(E)=\langle \exp(B+i\omega),v(E)\rangle$ and  heart $\ka(\omega,\beta)=\kp(0,1]$ constructed
as the tilt of $\coh(X)$ with respect to a certain torsion theory $(\kt_{(\omega,\beta)},\kf_{(\omega,\beta)})$.
Here, $\omega,B\in{\rm NS}(X)\otimes\QQ$, such that $\omega$ is ample with $\omega^2>2$, and
$\beta=(B.\omega)$. The class $$\exp(B+i\omega)=(1,B+i\omega,\frac{B^2-\omega^2}{2}+i(B.\omega))$$
is contained in $N(X)\otimes\QQ$, where $N(X)$ is the extended N\'eron--Severi lattice
$N(X)=H^*_{alg}(X,\ZZ)=\ZZ\oplus {\rm NS}(X)\oplus \ZZ$. Using the non-degenerate Mukai pairing, 
we shall tacitly identify $N(X)^*=\Hom(N(X),\CC)$ with $N(X)\otimes\CC$. In particular,
any stability function $Z$ can be written as $Z=\langle w,~~\rangle$ and thus
the map $\sigma=(Z,\kp)\mapsto Z$ is viewed as a map $${\rm Stab}(X)\to N(X)\otimes\CC.$$

Elements in $N(X)\otimes\CC$ of the form $\exp(B+i\omega)$ are contained in
a distinguished subset $\kp_0^+(X)$,
which shall be introduced next.

\begin{definition} Let $$\kp(X)\subset N(X)\otimes\CC$$
be the open set of all $\Omega\in N(X)\otimes\CC$ such
that ${\rm Re}(\Omega),{\rm Im}(\Omega)$ span
a positive plane in $N(X)\otimes\RR$.
\end{definition}

Here, $N(X)\otimes\RR$ is endowed with the Mukai pairing, i.e.\ the standard intersection pairing with an
extra sign on $H^0\oplus H^4$. In particular, $N(X)$ is a lattice of signature $(2,\rho(X))$.

\begin{ex}
For any $\omega\in {\rm NS}(X)\otimes\RR$ with $\omega^2>0$ and arbitrary $B\in {\rm NS}(X)\otimes\RR$,
the class $\exp(B+i\omega)$ is contained in $\kp(X)$. Indeed, ${\rm Re}(\exp(i\omega))=(1,0,-\omega^2/2)$
and ${\rm Im}(\exp(i\omega))=\omega$ are orthogonal of square $\omega^2$. Multiplication with
$\exp(B)$ is an orthogonal transformation of $N(X)\otimes\RR$, which proves the claim.
\end{ex}

The orientations of two given positive planes in $N(X)\otimes\RR$ can be compared via orthogonal projection.
This leads to a decomposition $$\kp(X)=\kp^+(X)\amalg\kp^-(X)$$ in two connected components.
We can distinguish one, say $\kp^+(X)$, by requiring that it contains $\exp(i\omega)$ with $\omega$ ample.

Instead of $\omega^2>2$ we assumed in Proposition \ref{prop:ZEsph} that
$Z(E)\not\in\RR_{\leq0}$ for all spherical $E\in\coh(X)$. The latter motivates a further shrinking of $\kp^+(X)$.

Let $\Delta:=\{\delta\in N(X)~|~\delta^2=-2\}$ which in particular contains the classical $(-2)$-classes
in ${\rm NS}(X)$ such as $[C]$ with $\PP^1\cong C\subset X$, but also $(1,0,1)$.
Then one defines $$\kp^+_0(X):=\kp^+(X)\setminus\bigcup_{\delta\in\Delta}\delta^\perp.$$
Since classes in $\Delta$ are real, the condition $\Omega\in\delta^\perp$ is equivalent to
$\langle{\rm Re}(\Omega),\delta\rangle=\langle{\rm Im}(\Omega),\delta\rangle=0$.

The following is the main result of \cite{BrK3}. 
Let ${\rm Stab}(X)^{\rm o}\subset{\rm Stab}(X)$ be the connected component of numerical
stability conditions containing those constructed in Corollary \ref{cor:constrstable}.

\begin{thm}\label{thm:coveringK3}
The natural projection $\sigma=(\kp,Z)\mapsto Z$ defines a topo\-logical covering
$$\xymatrix{{\rm Stab}(X)^{\rm o}\ar@{>>}[r]&\kp^+_0(X)}$$
with $\Aut^{\rm o}_0(\Db(X))$ as group of deck transformations.
\end{thm}
Here, $\Aut^{\rm o}_0(\Db(X))$ is the subgroup of $\Aut(\Db(X))$ of all linear
exact auto\-equivalences $\Phi:\Db(X)\congpf\Db(X)$ that preserve the distinguished component 
${\rm Stab}(X)^{\rm o}$ and such that $\Phi={\rm id}$ on $N(X)$.

\begin{conjecture}\label{conj:Brconj}{\bf (Bridgeland)} The component ${\rm Stab}(X)^{\rm o}$ is simply connected and, therefore,
$\pi_1(\kp^+_0(X))\congpf\Aut^{\rm o}_0(\Db(X))$.
\end{conjecture}

\begin{remark}
Stronger versions of this conjecture exist. E.g.\ one could conjecture that ${\rm Stab}(X)^{\rm o}$
is the only connected component (of maximal dimension) of ${\rm Stab}(X)$ or, slightly weaker,
that $\Aut(\Db(X))$ preserves
${\rm Stab}(X)^{\rm o}$.\footnote{In \cite{HMS} it was shown that $\Aut(\Db(X))$ at least preserves $\kp^+_0(X)$.}
 Both would imply that $$\Aut^{\rm o}_0(\Db(X))=\Aut_0(\Db(X))={\rm Ker}\left(\Aut(\Db(X))\to{\rm O}(H^*(X,\ZZ))\right)\!.$$ 
 
Since the image of the representation $\Aut(\Db(X))\to{\rm O}(H^*(X,\ZZ))$
has been described  (see \cite{HMS}), this would eventually yield a complete
description of $\Aut(\Db(X))$. Note that before Bridgeland phrased the conjecture above,
one had no idea how to describe $\Aut(\Db(X))$, even conjecturally. Describing the highly complex
group $\Aut(\Db(X))$ in terms of a certain fundamental group (see also Section \ref{sec:rephras}) is as explicit as 
a description can possibly get. It also fits well with Kontsevich's homological mirror symmetry which relates
$\Aut(\Db(X))$ of any Calabi--Yau variety $X$ with the fundamental group of the moduli space
of complex structures on its mirror \cite{KontENS}.
\end{remark}

Up to now, the conjecture has not been verified for a single projective K3 surface. The cases
of generic non-projective K3 surfaces and generically twisted projective K3 surfaces $(X,\alpha)$
have been successfully treated in \cite{HMSComp}. 
\subsection{Autoequivalences}\label{sec:autoexa} Before discussing some aspects of the proof of Theorem
\ref{thm:coveringK3}, it is probably useful to give a few explicit examples of auto\-equivalences of
$\Db(X)$ and to explain how  loops around $\delta^\perp$, the generators of $\pi_1(\kp^+_0(X))$,
can possibly be responsible for elements in $\Aut(\Db(X))$.

 As Mukai observed in \cite{Mu}, any $\Phi\in\Aut(\Db(X))$ induces an isometry $\Phi^H$ of $H^*(X,\ZZ)$ (as before, endowed with the usual intersection pairing
modified by a sign on $H^0\oplus H^4$) that also respects the weight two Hodge structure given by $H^{2,0}$.
By definition, $\Aut_0(\Db(X))=\{\Phi~|~\Phi^H={\rm id}\}$.

i) Any automorphism $f:X\congpf X$ induces naturally an autoequivalence:
$\Aut(X)\,\hookrightarrow \Aut(\Db(X))$, $f\mapsto f_*$. The induced action
on $H^*(X,\ZZ)$ is just the standard one.
\smallskip

ii) If $L\in\Pic(X)$, then $E\mapsto E\otimes L$ defines an autoequivalence. Thus,
$\Pic(X)\,\hookrightarrow\Aut(\Db(X))$. The induced action on $H^*(X,\ZZ)$ is described
by multiplication with $\exp({\rm c}_1(L))$.
\smallskip

iii) Recall that an object $E\in\Db(X)$ is called \emph{spherical} if $\Ext^*(E,E)$ is two-dimensional, i.e.\
$$\Ext^*(E,E)\cong
H^*(S^2,\CC).$$  

In particular, the Mukai
vector $v(E)\in N(X)$ of a spherical object $E$ is a $(-2)$-class, i.e.\ $v(E)\in\Delta$.
Any spherical object $E\in\Db(X)$ induces an autoequivalence $T_E$ which
is called the \emph{spherical twist} (or \emph{Seidel--Thomas twist}) associated with $E$. It can be described as the Fourier--Mukai transform
with Fourier--Mukai kernel given by the cone of the trace map
$E^*\boxtimes E\to\ko_\Delta$.

 The spherical twist sends an object $F\in\Db(X)$
to the cone of the evaluation map $\Ext^*(E,F)\otimes E\to F$. The important features
are: $T_E(E)\cong E[1]$ and $T_E(F)\cong F$ for any $F\in E^\perp$. For details
see e.g.\ \cite{HuyFM}. The induced action $T_E^H$ on cohomology
is given by reflection $s_{v(E)}$ in $v(E)^\perp$.  But note that $T_E$ itself is not of order two.
The reflections  $s_{v(E)}$ are well known classically for the case $E=\ko_C(1)$ where
$\PP^1\cong C\subset X$. Then $v(\ko_C(1))=(0,[C],0)$ and thus $T^H_{\ko_C(1)}=s_{v(\ko_C(1))}$ acts as
identity on $H^0\oplus H^4$ and as the reflection $s_{[C]}$ on $H^2$. The latter is
important for the description of the ample cone of a K3 surface.

The mysterious part of $\Aut(\Db(X))$ is the subgroup that is generated by spherical twists
$T_E$ and, very roughly, Bridgeland's conjecture sets it in relation to $\pi_1(\kp^+_0(X))$.
But even the construction of a homomorphism
\begin{equation}\label{eqn:fundAut}
\pi_1(\kp^+_0(X))\to\Aut(\Db(X)),
\end{equation}
one of the main achievements of \cite{BrK3}, is highly non-trivial. The naive idea that a loop
around $\delta^\perp$ is mapped to $T_E^2$, where $v(E)=\delta$, has among others
the obvious flaw that such an $E$ is not unique.
Theorem \ref{thm:coveringK3} not only says that (\ref{eqn:fundAut}) exists but determines
its image explicitly as $\Aut^{\rm o}_0(\Db(X))$.

\subsection{Building up ${\rm Stab}(X)^{\rm o}$}
Recall that in Section \ref{sec:constrhearts} stability conditions
$\sigma(\omega,B)$ on $\Db(X)$ were constructed with heart $\ka(\omega,\beta)$ and stability function
$Z=\langle\exp(B+i\omega),~~\rangle$. Here, $\beta=(B.\omega)$ with $B,\omega\in N(X)\otimes\RR$,
$\omega$ ample and such that $Z(\delta)\not\in\RR_{\leq0}$ for all $\delta\in\Delta$ of positive rank.\footnote{Strictly speaking, in Section \ref{sec:constrhearts} the result was only proved for rational such classes $B,\omega$. It is true without the rationality, but in \cite{BrK3} the proof is given after the analysis of ${\rm Stab}(X)^{\rm o}$.}
Let $$V(X)\subset{\rm Stab}(X)^{\rm o}$$ be the open subset of all these stability conditions.
As by construction a $\sigma(B,\omega)\in V(X)$ only depends on the stability function $Z=\langle\exp(B+i\omega),~~\rangle$, the projection ${\rm Stab}(X)\to N(X)^*=\Hom(N(X),\CC)\cong N(X)\otimes\CC$
identifies $V(X)$ with an open subset of $N(X)\otimes\CC$ and, in fact, under this identification $V(X)\subset\kp^+_0(X)$ (cf.\ \cite[Prop.\ 11.2]{BrK3}).

Applying the natural action of $\widetilde{\rm GL}\!{}^+(2,\RR)$ on ${\rm Stab}(X)$ (see
Section \ref{sect:GL}) yields more stability conditions and one introduces
$$U(X):=V(X)\cdot\widetilde{\rm GL}\!{}^+(2,\RR)\subset{\rm Stab}(X)^{\rm o},$$
which can in fact also be seen as a $\widetilde{\rm GL}\!{}^+(2,\RR)$-bundle over $V(X)$
(or rather its image in $\kp^+_0(X)$). Note that $V(X)$ can also
be considered naturally as a subset in  the quotients 
${\rm Stab}(X)^{\rm o}/\widetilde{\rm GL}\!{}^+(2,\RR)$ and $\kp^+_0(X)/{\rm GL}\!{}^+(2,\RR)$
which will be studied further in Section \ref{sec:rephras}.

As a consequence of the discussion in Section \ref{sec:Classi} one obtains the following characterization of
$U(X)$ (see \cite[Prop.\ 10.3]{BrK3}).

\begin{cor} A stability condition $\sigma\in{\rm Stab}(X)^{\rm o}$ is contained in $U(X)$ if and only
if all point sheaves $k(x)$ are $\sigma$-stable of the same phase.\qqed
\end{cor}
In fact, it would be enough to assume that $\sigma$ is contained in a good component, see
footnote page \pageref{good}.

The next step in the program is to study the boundary $\partial U(X)\subset{\rm Stab}(X)^{\rm o}$
and explain how to pass beyond it by applying spherical twists. Bridgeland studies two kinds of boundary components,
of type $A^\pm$ and of type $C$.

{\bf Boundary components of type $A$.} Consider a sequence $\sigma_t=(Z_t=\langle\exp(B_t+i\omega_t),~~\rangle,\kp_t)\in V(X)$ converging
to $\sigma=\sigma_0=(Z=\langle\exp(B+i\omega),~~\rangle,\kp)\in\partial U(X)\subset{\rm Stab}(X)^{\rm o}$.
One reason for $\sigma$ not being in $V(X)$ is that there exists a $\delta=(r,\ell,s)\in\Delta$ of positive rank
with $Z(\delta)\in\RR_{\leq0}$ and, since $\sigma$ is in the boundary, in fact $Z(\delta)=0$ (cf.\ Proposition
\ref{prop:ZEsph}).
The imaginary part of the equation reads $\mu(\delta):=(\ell.\omega)/r=(B.\omega)$ and for fixed
$B$ this defines a decomposition of the positive cone into the two regions $\mu>(B.\omega)$ and $\mu<(B.\omega)$:
\vskip-1cm
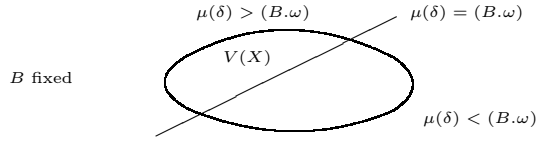
$$\begin{picture}(-60,100)
\put(2,30){\qbezier(-70.3,20.3)(-33.3,40.3)(7.1,19.5)}
\put(2,30){\qbezier(-70.3,20.3)(-87.3,10.3)(-70.3,0.35)}
\put(2,30){\qbezier(-70.3,0.3)(-30.3,-16.3)(6.7,-0.4)}
\put(2,30){\qbezier(7,19.6)(21.3,9.3)(6.7,-0.45)}
\put(-80,20){\line(2,1){90}}
\put(15,65){\tiny\mbox{$\mu(\delta)=(B.\omega)$}}
\put(-65,65){\tiny\mbox{$\mu(\delta)>(B.\omega)$}}
\put(20,25){\tiny\mbox{$\mu(\delta)<(B.\omega)$}}
\put(-55,48){\tiny\mbox{$V(X)$}}
\put(-135,40){\tiny\mbox{$B~\text{fixed}$}}
\end{picture}$$
However, not all of $\mu(\delta)=(B.\omega)$ is part of the boundary. Indeed,
the real part of $Z(\delta)$  can be written
as $(B.\ell)-s-r\frac{B^2-\omega^2}{2}$. So again for fixed $B$, the boundary looks more like this
(with two holes defined by $\delta^\perp$):
\vskip-1cm
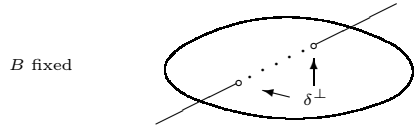
$$\begin{picture}(-60,100)
\put(2,30){\qbezier(-70.3,20.3)(-33.3,40.3)(7.1,19.5)}
\put(2,30){\qbezier(-70.3,20.3)(-87.3,10.3)(-70.3,0.35)}
\put(2,30){\qbezier(-70.3,0.3)(-30.3,-16.3)(6.7,-0.4)}
\put(2,30){\qbezier(7,19.6)(21.3,9.3)(6.7,-0.45)}
\put(-80,20){\line(2,1){30}}
\put(-45,37.5){\circle*{1}}
\put(-40,40){\circle*{1}}
\put(-35,42.5){\circle*{1}}
\put(-30,45){\circle*{1}}
\put(-25,47.5){\circle*{1}}
\put(-20,50){\line(2,1){30}}
\put(-21,49.4){\circle{2}}
\put(-30,30){\vector(-4,1){10}}
\put(-21,34.5){\vector(0,1){10}}
\put(-49.10,35.5){\circle{2}}
\put(-25,27){\tiny\mbox{$\delta^\perp$}}
\put(-135,40){\tiny\mbox{$B~\text{fixed}$}}
\end{picture}$$
\begin{remark}
Here is what this means in terms of a spherical bundle $A$ with $v(A)=\delta$. Suppose $A$ is
$\mu_{\omega_t}$-stable.  Since $\mu_{\omega_t}(A)>(B.\omega)$, the bundle $A$ is an object 
in $\ka(\omega_t,\beta_t)$ for $t>0$. As $\omega_t$ crosses the wall $\mu=(B.\omega)$,
the bundle $A$ can obviously not be contained in the heart of the stability condition any longer.
However, if $\omega_t$ crosses the dotted part, then $A[1]$ will, but something more drastically
happens, when the solid part of $\mu=(B.\omega)$ is crossed.

The set $U(X)$ is determined by the stability of the point sheaves $k(x)$. Every bundle $A$ as before
admits a non-trivial homomorphism $A\to k(x)$ to any $k(x)$.  For $\omega_t$ passing through the solid part
of $\mu=(B.\omega)$, the object
$A$ becomes an object in $\kp(1)$ and the existence of the non-trivial $A\to k(x)$ shows that $k(x)$ can no longer
be stable and, therefore, $\sigma\not\in U(X)$. On the other hand, passing through the dotted part the object $A$ becomes
semistable of phase $0$ and, in particular, the non-trivial $A\to k(x)$ does not contradict stability of $k(x)$.
Thus, $\sigma$ is still in $U(X)$.\footnote{The discussion is simplified by assuming that $\omega_t$ stays ample when passing through $\mu=(B.[C])$.} The image under the stability function looks likes this:
\vskip-1cm
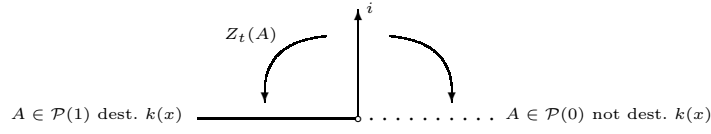
$$\begin{picture}(-40,100)
\put(0,31){\vector(0,1){40}}
\put(-60,30){\line(1,0){59}}
\put(5,30){\circle*{1}}
\put(10,30){\circle*{1}}
\put(15,30){\circle*{1}}
\put(20,30){\circle*{1}}
\put(25,30){\circle*{1}}
\put(30,30){\circle*{1}}
\put(35,30){\circle*{1}}
\put(40,30){\circle*{1}}
\put(45,30){\circle*{1}}
\put(50,30){\circle*{1}}
\put(0,30){\circle{2}}
\put(3,70){\tiny\mbox{$i$}}
\put(-50,60){\tiny\mbox{$Z_t(A)$}}
\put(-130,30){\tiny\mbox{$A\in\kp(1)$\ dest.\ $k(x)$}}
\put(55,30){\tiny\mbox{$A\in\kp(0)$\ not dest.\ $k(x)$}}
\put(-5,50){\qbezier(-7,11)(-30,9)(-30,-10)}
\put(5,50){\qbezier(7,11)(30,9)(30,-10)}
\put(35,37){\vector(0,-1){}}
\put(-35,37){\vector(0,-1){}}
\end{picture}$$

\end{remark}

{\bf Boundary components of type $C$.} So far we have not taken into account the condition that $\omega$ has to be ample. It is well known that the ample cone
inside the positive cone $\kc_X\subset {\rm NS}(X)\otimes\RR$ is cut out by the condition $(\omega.[C])>0$ for all
$(-2)$-curves $\PP^1\cong C\subset X$. Suppose now that in fact $\omega$ is still contained
in the positive cone $\kc_X$ but in the part of the boundary of the
ample cone described by $(\omega.[C])=0$. Then $\exp(B+i\omega)\in\kp^+_0(X)$ implies
$\langle\exp(B+i\omega),(0,[C],k)\rangle\ne0$ for all $k\in \ZZ$. Thus, on the line of all $(B.[C])$ the boundary of
$\partial\kp_0^+(X)$ looks like
\vskip-1cm
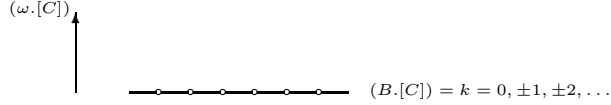
$$\begin{picture}(0,100)
\put(-70,30){\vector(0,2){30}}
\put(-95,60){\tiny\mbox{$(\omega.[C])$}}
\put(-50,30){\line(1,0){10}}
\put(-39,30){\circle{2}}
\put(-38,30){\line(1,0){10}}
\put(-27,30){\circle{2}}
\put(-26,30){\line(1,0){10}}
\put(-15,30){\circle{2}}
\put(-14,30){\line(1,0){10}}
\put(-3,30){\circle{2}}
\put(-2,30){\line(1,0){10}}
\put(9,30){\circle{2}}
\put(10,30){\line(1,0){10}}
\put(21,30){\circle{2}}
\put(22,30){\line(1,0){10}}
\put(40,29){\tiny\mbox{$(B.[C])=k=0,\pm1,\pm2,\ldots$}}
\end{picture}$$
\vskip-0.5cm
\begin{remark}
Recall, that the torsion sheaves $\ko_C(k-1)$ (with Mukai vector $(0,[C],k)$)
are torsion sheaves on $X$ and thus contained
in $ \kt_{(\omega_t,\beta_t)}\subset\ka(\omega_t,\beta_t)$. What happens with them when $\omega_t$ passes
through the wall defined by $[C]$? This depends on the position of $B$. More precisely,
when passing through the ray $(B.[C])> k$ the real part of $Z(\ko_C(k-1))$ is po\-si\-tive and,
therefore, $\ko_C(k-1)$ (and in fact all $\ko_C(\ell)$ with $\ell\geq k-1$) can no longer be in the heart of $\sigma$, they have to be shifted. However, passing through $(B.[C])<k$ a priori does not affect the stability of $\ko_C(k-1)$. But the stability of
$\ko_C(k-1)$ is not what matters for the description of the boundary of $U(X)$. Instead, one has to
study the stability of all point sheaves $k(x)$ as before. For $x$ contained in $C$, there exists a non-trivial $\ko_C(k)\to k(x)$ and depending on whether one passes through $(B.[C])<k$ or $(B.[C])>k$ this affects the stability of $k(x)$ or not.
However, if $x\not\in C$, then the stability of $k(x)$ is not affected by the process at all. Thus, the boundary of $U(X)$ that
is caused by non-trivial $\ko_C(k-1)\to k(x)$ looks like this
\vskip-1.8cm
$$\begin{picture}(0,100)
\put(20,30){\circle*{1}}
\put(30,30){\circle*{1}}
\put(40,30){\circle*{1}}
\put(50,30){\circle*{1}}
\put(60,30){\circle*{1}}
\put(-50,30){\line(1,0){60}}
\put(11,30){\circle{2}}
\put(-40,15){\tiny\mbox{$(B.[C])<k$}}
\end{picture}$$
but for varying $k$ they fill up a straight line with all $(0,[C],k)^\perp$ removed.
\end{remark}
%

The two types of boundary above, type $A$ and type $C$, are fundamentally different. Their images in $\kp_0^+(X)$
look like a real boundary for type $C$ and like a ray removed for type $A$:

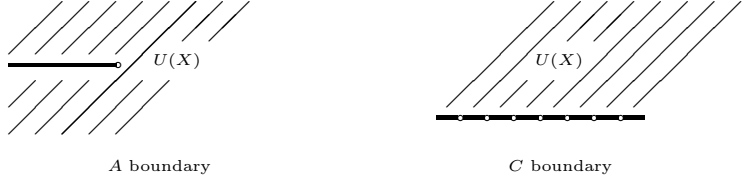
$$\begin{picture}(0,100)
{\linethickness{0.5mm}
\put(-122,60){\line(1,0){40}}}
\put(-122,64){\line(1,1){20}}
\put(-112,64){\line(1,1){20}}
\put(-102,64){\line(1,1){20}}
\put(-102,34){\line(1,1){50}}
\put(-122,34){\line(1,1){20}}
\put(-81,60){\circle{2}}
\put(-122,44){\line(1,1){10}}
\put(-112,34){\line(1,1){20}}
\put(-102,34){\line(1,1){50}}

\put(-92,34){\line(1,1){20}}
\put(-82,34){\line(1,1){20}}
\put(-92,64){\line(1,1){20}}
\put(-82,64){\line(1,1){20}}
\put(-47,69){\line(1,1){15}}
\put(-57,69){\line(1,1){15}}
\put(-68,60){\tiny\mbox{$U(X)$}}
\put(-85,20){\tiny\mbox{$A$\ boundary}}
\put(42,44){\line(1,1){40}}
\put(52,44){\line(1,1){40}}
\put(62,44){\line(1,1){10}}
\put(87,69){\line(1,1){15}}
\put(72,44){\line(1,1){10}}
\put(97,69){\line(1,1){15}}
\put(92,44){\line(1,1){40}}
\put(82,44){\line(1,1){40}}
\put(112,44){\line(1,1){40}}
\put(102,44){\line(1,1){40}}
\put(75,60){\tiny\mbox{$U(X)$}}
\put(65,20){\tiny\mbox{$C$\ boundary}}
 {\linethickness{0.5mm}
\put(38,40){\line(1,0){8}}
\put(48,40){\line(1,0){8}}
\put(58,40){\line(1,0){8}}
\put(68,40){\line(1,0){8}}
\put(78,40){\line(1,0){8}}
\put(88,40){\line(1,0){8}}
\put(98,40){\line(1,0){8}}
\put(108,40){\line(1,0){8}}}
\put(47,40){\circle{2}}
\put(57,40){\circle{2}}
\put(67,40){\circle{2}}
\put(77,40){\circle{2}}
\put(87,40){\circle{2}}
\put(97,40){\circle{2}}
\put(107,40){\circle{2}}

\end{picture}$$
The above oversimplified discussion was meant to motivate the following result \cite[Thm.\ 12.1]{BrK3}. The details
of the proof are quite lengthy and the case ii) below did not appear above.

\begin{prop}\label{prop:boundary}
For a general point of the boundary $\sigma=(Z,\kp)\in\partial U(X)$ one of the following possibilities occur:

i) The stable factors of $k(x)$, $x\in X$,  are isomorphic to a spherical bundle $A$ of rank $r$ and to
$T_A(k(x))$. More precisely, there exists an exact sequence in $\kp(\phi(k(x)))$ of the form
$0\to A^{\oplus r}\to k(x)\to T_A(k(x))\to 0$.

ii) The stable factors of $k(x)$, $x\in X$,  are isomorphic to $A[2]$ with $A$ a spherical bundle of rank $r$ and to
$T_A^{-1}(k(x))$. There exists an exact sequence in $\kp(\phi(k(x)))$ of the form
$0\to T_A^{-1}(k(x))\to k(x)\to A^{\oplus r}[2]\to 0$.

iii) There exists a curve $\PP^1\cong C\subset X$ such that all $k(x)\in\kp(1)$ are stable for $x\not\in C$ and
have stable factors $\ko_C(k+1)$ and $\ko_C(k)[1]$ for some $k$, i.e.\ there exists an exact sequence
$0\to\ko_C(k+1)\to k(x)\to\ko_C(k)[1]\to 0$.
\end{prop}

The picture above of the images in $\kp^+_0(X)$ of the two kinds of boundary of $U(X)$ also suggests
how to get back into $U(X)$ by autoequivalences. Suppose $\sigma_t$, $t\in(-\varepsilon,\varepsilon)$, with $\sigma_{t>0}\in U(X)$ and $\sigma_{t<0}\not\in U(X)$.
For a type $C$ boundary, a reflection in $(0,[C],k)^\perp$ which lifts to the spherical twist $T_{\ko_C(k)}$ interchanges the two sides of the boundary, i.e.\ $T_{\ko_C(k)}(\sigma_{t<0})\in U(X)$.
This is clear from the picture in $\kp^+_0(X)$ and is proved in \cite[Sect.\ 13]{BrK3} using
Proposition \ref{prop:boundary}. For a type $A$ boundary, one uses the square $T_A^2$, i.e.\
$T_A^2(\sigma_{t<0})\in U(X)$.

\begin{remark} Thus, more precisely Bridgeland shows
$$\Aut_0^{\rm o}(\Db(X))=\langle T_{\ko_C(k)},T_A^2\rangle\cap \Aut_0(\Db(X)),$$ where all $\PP^1\cong C\subset X$
and all $k\in \ZZ$ really occur. Which of the spherical bundles $A$ really occur is not known. In particular, it is not clear if every spherical bundle on $X$ is $\mu$-stable with respect to at least
one ample $\omega$. For $\Pic(X)\cong \ZZ$ every spherical bundle is $\mu$-stable, see \cite{Mu}. 

Building up ${\rm Stab}(X)^{\rm o}$ from $U(X)$ eventually allows Bridgeland to deduce
the existence of a  group homomorphisms
$$\pi_1(\kp^+_0(X))\to\langle T_{\ko_C(k)},T_A^2\rangle\subset\Aut(\Db(X)).$$
\end{remark}

Our discussion does not do justice to \cite{BrK3}, it can at most serve as an illustration. Many
of the intricacies have not been mentioned. E.g.\ why is $\pi({\rm Stab}(X)^{\rm o})$ exactly
$\kp^+_0(X)$? Why is $\pi:{\rm Stab}(X)^{\rm o}\to\kp_0^+(X)$ really a topological cover?
We also have not explained how to use the analysis of ${\rm Stab}(X)^{\rm o}$ to finish the proof of Corollary \ref{cor:constrstable} for non rational $\omega,B$.

Related questions of this type are addressed in \cite{Hart}. For example it is shown that
$T_A$ for any Gieseker stable spherical bundles $A$ preserves the distinguished
component. The special role of spherical objects in the study of ${\rm Stab}(X)^{\rm o}$ is further
discussed in \cite{HMZ}.
\subsection{Moduli space rephrasing}\label{sec:rephras}
Instead of $\kp^+_0(X)$ one can consider the classical
period domain $$D\subset\PP(N(X)\otimes\CC)$$
of all $x\in\PP(N(X)\otimes\CC)$ with $(x.x)=0$ and $(x.\bar x)>0$, which 
is an open set of the smooth quadric defined by $(x.x)=0$. Its two connected components
$D^\pm\subset D$ are interchanged by complex conjugation. Similarly to the
definition of $\kp^+_0(X)$ one sets $$D_0:=D\setminus\bigcup_{\delta\in\Delta}\delta^\perp$$
and  $D^\pm_0:= D^\pm\cap D_0$.

 The natural projections $\kp(X)\to D$ and
$\kp_0^+(X)\to D_0^+$ are ${\rm GL}\!{}^+(2,\RR)$-bundles and the Serre spectral sequence
yields an exact sequence
$$\ZZ\cong\pi_1({\rm GL}\!{}^+(2,\RR))\to\pi_1(\kp^+_0(X))\to\pi_1(D_0^+)\to1.$$
Conjecture \ref{conj:Brconj} is therefore equivalent to $$\pi_1(D_0^+)\cong\raisebox{0.5mm}
{$\Aut_0^{\rm o}(\Db(X))$}/\raisebox{-1mm}{$\ZZ[2]$}.$$ Note however that $\pi_1(\kp^+_0(X))$ and $\pi_1(D_0^+)$ are usually not finitely gene\-ra\-ted
which makes them slightly unpleasant to work with. But they can be related to the fundamental
group of a quasi-projective variety, which is of course finitely generated, as follows.

Period domains of the type $D_0$ as above are well studied in moduli theory of K3 surfaces.
Dividing out by the subgroup of the orthogonal group ${\rm O}(H^2(X,\ZZ))$ fixing the polarization would yield the moduli space of polarized K3 surfaces. In analogy, one considers here
the subgroup $$\Gamma:=\{g\in {\rm O}(N(X))~|~g_{N(X)^*/N(X)}={\rm id}\}$$ which coincides with
the group of orthogonal transformations of $N(X)$ that can be extended to an orthogonal transformation of  $H^*(X,\ZZ)$ acting trivially on the transcendental lattice. 
The lift of $\Gamma$ to $\Aut(\Db(X))$ is the group of symplectic autoequivalences
$$\Aut(\Db(X))_{\rm s}:=\{\Phi\in\Aut(\Db(X))~|~\Phi^H|_{T(X)}={\rm id}\}.$$

Restricting to $\Aut^{\rm o}(\Db(X))_{\rm s}=\Aut^{\rm o}(\Db(X))\cap\Aut(\Db(X))_{\rm s}$ yields
a group that acts naturally on the distinguished component ${\rm Stab}(X)^{\rm o}$.
Then $$\raisebox{-1mm}{$\Aut^{\rm o}(\Db(X))_{\rm s}$}\setminus\raisebox{0.5mm}{${\rm Stab}(X)^{\rm o}$}/\raisebox{-1mm}{$\widetilde{\rm GL}\!{}^+(2,\RR)$}\cong \raisebox{0.5mm}{$D_0$}/\raisebox{-1mm}{$\Gamma$}.$$ 

Thus, the (distinguished component of the) space of stability conditions ${\rm Stab}(X)$ leads naturally, by identifying stability conditions which only differ
by $\widetilde{\rm GL}\!{}^+(2,\RR)$ and $\Aut^{\rm o}(\Db(X))_{\rm s}$,
to the quasi-projective variety $D_0/\Gamma$. At this point one starts wondering whether there
is a more functorial approach towards stability condition that awaits to be unraveled. 

In any case, this point of view allows one to rephrase Bridgeland's origi\-nal conjecture as
a conjecture that only involves
finitely generated groups.

\begin{conjecture}
There exists a natural isomorphism
$$\raisebox{0.5mm}{$\Aut(\Db(X))_{\rm s}$}/\raisebox{-1mm}{$\ZZ[2]$}\cong\pi_1^{\rm st}([\raisebox{0.5mm}{$D_0$}/\raisebox{-1mm}{$\Gamma$}]).$$
\end{conjecture}

While the difference between the stacky fundamental group
$\pi_1^{\rm st}([D/\Gamma])$ and the ordinary $\pi_1(D/\Gamma)$ is huge,
for the stack $[D_0/\Gamma]$ the two only differ by the usually very small subgroup of $\Gamma$
of elements with fixed points in $D_0$ which sometimes is even trivial. In any case, one always has an exact sequence $1\to\pi_1(D_0)\to\pi_1^{\rm st}([D_0/\Gamma])\to\Gamma\to\ZZ/2\ZZ\to0$.

\begin{remark}
There is another way of looking at this picture. The period domains $D_0$ and their quotients
$[D_0/\Gamma]$ are well studied spaces. Essentially, $D_0$ is the complement of the infinite
hyperplane arrangement $\bigcup \delta^\perp$. The universal cover of such spaces can rarely be
described explicitly and/or in a meaningful way (e.g.\ with a moduli theoretic interpretation). But the
mo\-duli space of stability conditions ${\rm Stab}(X)^{\rm o}/\widetilde{\rm GL}\!{}^+(2,\RR)$
provides such a  description provided it really is simply connected as predicted by
Bridgeland's conjecture. Once this is settled, one would ask whether
it is actually contractible which would make $D_0$ a $K(\pi,1)$-space. Questions of this type have
recently also been addressed by Allcock in \cite{All}.
\end{remark}

Geodesics in $D_0$ converging to cusps, which are in bijection with
Fourier--Mukai partners of $X$ (see \cite{SMa}),
have been studied in detail in \cite{Hartm}. They are shown
to be related to so-called linear degenerations in ${\rm Stab}(X)^{\rm o}$.

\section{Further results}
There was not enough time in the lectures nor is there enough space here to enter the description of $\Stab(\kd)$ for
other situations, only the case of curves and (K3) surfaces has been discussed in some detail.
In fact, there are not many other examples of the type $\kd=\Db(X)$ with $X$ smooth and projective
one really understands well. In particular, Calabi--Yau threefolds remain elusive.
But there are examples of non-compact $X$ and, closely related, more algebraic triangulated categories
$\kd$ for which the theory is quite well understood, see e.g.\
\cite{Brsurv,Wolf1} for more information.

\subsection{Non-compact cases}
i) The derived category $\Db(X)$ of a K3 surface $X$ is an example
of a \emph{K3 category}, i.e.\ a triangulated category for which the double shift $[2]$
is a Serre functor. As we have seen, a description of ${\rm Stab}(\Db(X))$ or
$\Aut(\Db(X))$ is non-trivial (and still not completely understood).
The case of local K3 surfaces is more accessible. More precisely,
for the minimal resolution $\pi:X\to\CC^2/G$ of a Kleinian singularity one can consider K3
categories $\kd\subset\hat\kd\subset\Db(X)$ of complexes supported on the exceptional
divisor (resp.\ with vanishing $R\pi_*$). The spaces $\Stab(\kd)$ and $\Stab(\hat\kd)$ 
are studied in detail in \cite{Ishi,Th} for  $A_n$-singularities and in \cite{Brav,BKl} in general.
The analogue of Bridgeland's conjecture, originally  formulated for
projective K3 surfaces, has been proved in the local situation for the category $\kd$.

ii) Triangulated categories with a Serre functor given by the triple shift $[3]$ are called
\emph{CY3-categories}.
The most prominent examples is $\Db(X)$ with $X$ a smooth projective Calabi--Yau threefold. More accessible examples are provided by local Calabi--Yau manifolds. For example, if $X$ is the total space of
$\omega_{\PP^2}$, then the bounded derived category  $\kd=\kd_0(X)$ of coherent sheaves on $X$
supported on the zero section of $X\to\PP^2$ is a CY3-category. The study of ${\rm Stab}(\kd_0(X))$
has been initiated in \cite{Brnc} and a complete connected component of ${\rm Stab}(\kd_0(X))$
has been described in great detail in \cite{BM}. In particular, it is shown to be simply connected, which
in the case of $\Db(X)$ of a compact K3 surface is still conjectural (see Conjecture \ref{conj:Brconj}).
Categories of this type, i.e.\ derived categories of coherent sheaves supported on the zero section
of the canonical bundle of a Fano surface, often allow for a more combinatorial approach via
quiver representations, see \cite{Brsurv} for a quick introduction.

\subsection{Compact cases} i) Generic non-projective K3 surfaces $X$ and generic twisted but projective
K3 surfaces $(X,\alpha)$ have been studied in \cite{HMSComp} (see \cite{Ok2} for related results).
As in these cases there are no 
spherical objects, the space of stability conditions is much easier to describe.
Stability conditions on $\Db(Y)$ of an Enriques surface are in \cite{MMS} related to
stability conditions on the covering K3 surface $X$. In fact, for generic $Y$ it is shown that
a connected component of ${\rm Stab}(Y)$ is isomorphic to the
component ${\rm Stab}(X)^{\rm o}$ discussed in the earlier sections.

ii) In \cite{BMT} the authors come close to constructing stability conditions on $\Db(X)$ for
projective Calabi--Yau threefolds or, in fact, arbitrary projective threefolds.  The
remaining problem is related to certain Chern class inequalities which have been
further studied in \cite{BBMT}.

iii) Stability conditions on $\Db(\PP^n)$ are more accessible due to the existence of full exceptional collections.
In particular the case of $\PP^1$ is completely understood, see \cite{Ma,Ok}. For $n>1$ see \cite{Maold}, where
one also finds results on Del Pezzo surfaces.



\end{document}